\documentclass[10pt]{article}
\usepackage{amsmath,amssymb,amsthm,amscd,esint}

\newtheorem{theorem}{Theorem}[section]
\newtheorem{lemma}[theorem]{Lemma}

\newtheorem{corollary}[theorem]{Corollary}
\newtheorem{proposition}[theorem]{Proposition}

\theoremstyle{definition}
\newtheorem{definition}[theorem]{Definition}

\theoremstyle{remark}

\numberwithin{equation}{section}
\numberwithin{theorem}{section}

\DeclareMathOperator{\tr}{tr}

\DeclareMathOperator{\Ric}{Ric}

\begin{document}

\title{Convergence of Scalar Curvature of Long Time K\"ahler-Ricci Flow on K\"ahler Manifold}

\author{Lei Zhang\thanks{Supported by Postdoctoral Fellowship Program GZC20240867. Email: leizizhang1992@163.com}\\
School of Mathematics and Statistics, Jiangsu Normal University\\
Zhenlei Zhang\thanks{Supported partially by NSFC 11431009, NSFC 11771301 and NSFC 12531001. Email: zhleigo@aliyun.com}\\
School of Mathematical Sciences, Capital Normal University}
\date{}

\maketitle



\begin{abstract}
This paper is concerned with the long time K\"ahler-Ricci flow on a compact K\"ahler manifold $X$ of complex dimension $n$ with semi-ample canonical bundle. We establish a uniform $\mu$-entropy estimate and a uniform Sobolev inequality along the normalized K\"ahler-Ricci flow. As a consequence, we prove that the scalar curvature of the K\"ahler metrics along the normalized K\"ahler-Ricci flow converges uniformly on $X$ to the negative Kodaira dimension $-\kappa$ of $X$. In the noncollapsing case $\kappa=n$ we also obtain the rate of convergence $-Ce^{-t}\leq R(t)+n\leq C(1+t)^{-1}$ and $\fint_X|R(t)+n|\omega(t)^n\leq Ce^{-t}$.
\end{abstract}

\tableofcontents



\section{Introduction}
The Analytic Minimal Model Program through K\"ahler-Ricci flow was initiated by Song-Tian and their collaborators \cite{ST07,ST12,ST17,SW13,TZh06}. Many progresses have been made on the program. The behavior of long-time solutions of the K\"ahler-Ricci flow has been extensively studied \cite{Cao85,GTZ20,HLT24+,JS22,LTZ26+,ST07,ST12,ST16,ST17,STZ19,TZhaZ16,TWY18,TosZh15,Tsu88,Wan18,ZhY19} after the pioneering work of Song-Tian \cite{ST07,ST12} in the framework of the analytic minimal model program with Ricci flow \cite{ST17}. There has since been extensive work done to better understand the behavior of the K\"ahler-Ricci flow in the finite-time singularities in \cite{Bam18,CW12,CW20,HJST24,JST23+,SesT08,SSW13,SW13,TZhZ16,TosZh18,ZhaZh23} and references therein. For a comprehensive overview of these developments, readers are referred to the surveys by Song-Weinkove \cite{SW13-1} and Tosatti \cite{To18}.

Let $(X^n,\omega_0)$ be a compact K\"ahler manifold of complex dimension $n$. The K\"ahler-Ricci flow on $X$ defined by
\begin{equation}\label{KRF}
  \frac{\partial\tilde{\omega}}{\partial t}=-Ric(\tilde{\omega}),\qquad \tilde{\omega}(0)=\omega_0.
\end{equation}
The flow has a solution for short positive time \cite{Cao85} and the maximal existence time $T$ of (\ref{KRF}) is uniquely determined by a result of Tian-Zhang in \cite{TZh06}, namely,
$$T=\sup\{t>0:~[\omega_0]-2\pi tc_1(X)>0\}.$$
The K\"ahler classes $[\tilde{\omega}(t)]$ at time $t$ is precisely given by $[\omega_0]-2\pi tc_1(X).$ It is well-known that the K\"ahler-Ricci flow (\ref{KRF}) admits a long-time solution if and only if the canonical line bundle $K_X$ is nef. The abundance conjecture predicts that if the canonical line bundle $K_X$ over a projective manifold $X$ is nef, then it must be semi-ample, i.e., a sufficiently large power of $K_X$ is globally generated or base point free. The deep and subtle relationship between these two notions of positivity in algebraic geometry is also reflected in the canonical metric structures of the underlying K\"ahler manifolds.

In this paper, we focus on the normalized K\"ahler-Ricci flow on an $n$-dimensional K\"ahler manifold $X$ with semi-ample canonical line bundle defined by
\begin{eqnarray}\label{UKRF}
\frac{\partial\omega}{\partial t}=-Ric(\omega)-\omega,\qquad \omega(0)=\omega_0.
\end{eqnarray}
Let $f:X\to X_{can}\subset\mathbb{P}^N$ be the canonical map induced by the pluricanonical system $|mK_X|$ for sufficiently large $m\in\mathbb{Z}^+,$ where $X_{can}$ is the canonical model of $X.$ We assume that the Kodaira dimension $\kappa$ belongs to $\{1,2,\cdots,n-1\}$ and so $X$ is an algebraic fiber over space over $X_{can}.$ Let
$$X_{can}^o=\{y\in X_{can}|~y~\mbox{is~a~nonsingular~fiber~and}~X_y=f^{-1}(y)~\text{is~a nonsingular~fiber}\}$$
and $X^o=f^{-1}(X_{can}^o)$, $D=X_{can}\setminus X_{can}^o$.

When $0<\kappa<n$, Song-Tian \cite{ST07,ST12} proved that the flow collapses the nonsingular Calabi-Yau fibers. Its weak limit is the pullback of a metric $\omega_{can}$ which is smooth on $X_{can}^o$ and satisfies
$$\Ric(\omega_{can})=-\omega_{can}+\omega_{WP},$$
where $\omega_{WP}$ is Weil-Petersson metric induced by the Calabi-Yau fibration $f$. They also proved the $C^0$-convergence on the potential level \cite{ST12} and in the case when $X$ is an elliptic surface the $C_{loc}^{1,\alpha}$-convergence of potentials on $X^o$ for any $\alpha<1$ \cite{ST07}. The canonical metric $\omega_{can}$ extends to a unique twisted K\"ahler-Einstein current with bounded local potentials, and $\omega_{WP}$ also extends to a closed positive current on $X_{can}$
\cite{GS21}. The $C^{1,\alpha}$-convergence of potentials when $X$ is a global submersion over $X_{can}$, as well as Gromov-Hausdorff convergence in a special case, was proved by Gross-Tosatti-Zhang \cite{GTZ13,GTZ16}, Hein-Tosatti \cite{HeTo15} and Fong-Zhang \cite{FoZh15}. In \cite{TWY18}, Tosatti-Weinkove-Yang improved the estimate and showed that the metric $\omega(t)$ converges to $f^*\omega_{can}$ in $C_{loc}^{0}$ on $X^o$. It was shown in \cite{CLee23} that $\omega(t)\to f^*\omega_{can}$ in $C_{loc}^{\alpha}(X^o)$ as $t\to\infty$ for any $\alpha<1$. Moreover, Tosatti-Weinkove-Yang \cite{TWY18} proved that the restricted metric $\omega(t)|X_y$ converges (up to scalings) in the $C^0$-topology to the unique Ricci flat metric in the class $[\omega_0|X_y]$ on the fibre $X_y$ for any regular value $y$. This result is improved to be smooth convergence by Tosatti-Zhang in \cite{TosZh15}. A clearer and more unified exposition can also be found in Tosatti's note \cite[Section 5]{To18}.

Tian-Zhang \cite{TZhaZ16} applied their $L^4$-norm of Ricci curvature to show that on any minimal threefold of general type, the normalized K\"ahler-Ricci flow converges to the unique generalized K\"ahler-Einstein metric on the canonical model in the Cheeger-Gromov topology. The diameter bound is proved for minimal models of general type in \cite{Wan18}. In general, Fu-Guo-Song \cite{FGS20} proved that $(X_{can}^o,\omega_{can})$ has bounded diameter, and Song-Tian-Zhang \cite{STZ19} established that its metric completion is a compact metric space. In the special case when $dim X=2$ or the general fibre of $X$ over $X_{can}$ is a complex torus, \cite{STZ19} established that the diameter is uniformly bounded, via an application of Tian-Zhang's relative volume comparison \cite{TZhZ21}. In \cite{GTZ20}, Gross-Tosatti-Zhang generalized the results in \cite{STZ19} in the case when $X_{can}$ is smooth and the union of all codimension $1$ irreducible components of $D$ is snc, and assuming also a locally uniform Ricci curvature bound away from the singular fibers of $f$, as in \cite{STZ19}. To apply such a relative volume comparison \cite{TZhZ21}, one needs to obtain a uniform bound for the Ricci curvature in a neighborhood of a nonsingular fibre of $X$ over $X_{can}$. It was recently shown in \cite{HLT24+} that $\omega(t)\to f^*\omega_{can}$ in $C_{loc}^{\infty}(X^o)$ as $t\to\infty$ and with bounded Ricci curvature away from the singular fibers. In \cite{JS22}, Jian-Song applied Bamler's Harnack inequality for Nash entropy \cite{Bam20} to derive a relative volume comparison for the Ricci flow, eliminating the need for a locally uniform Ricci curvature bound away from the singular fibers. In the same paper, Jian-Song further proved that the normalized K\"ahler-Ricci flow on a minimal threefold converges sequentially in Gromov-Hausdorff topology to a compact metric space homeomorphic to its canonical model. Then, by incorporating the techniques from \cite{STZ19}, Jian-Song also established a uniform diameter bound for $X$ along the normalized K\"ahler-Ricci flow \cite{JS22}.

Let us remark that the uniform diameter bound and the existence of subsequential Gromov-Hausdorff limits of $\big(X, \omega(t)\big)$ were recently achieved \cite{GuPhSoSt24-1,ZhZh25-2} under the assumption that $K_X$ is nef. Lastly, Sz\'ekelyhidi \cite{Sz25,Sz25+} (see also \cite[Theorem~2.1]{LTZ26+}) has shown that for $0<\kappa\leq n$ the metric completion of $\big(X_{can}^o,\omega_{can}\big)$ is homeomorphic to $X_{can}$, is a non-collapsed $\mathrm{RCD}(-1,2\kappa)$-space (equipped with the measure $\omega_{can}^{\kappa}$), and $X_{can}\setminus X_{can}^o$ has real Hausdorff codimension at least $2$.

In \cite{ST16}, Song-Tian showed that the scalar curvature is uniformly bounded on $X\times[0,\infty)$ along the normalized K\"ahler-Ricci flow \eqref{UKRF}. The case when $X$ is of general type is given by Zhang in \cite{ZhZ09}. When $0<\kappa<n$, Jian \cite{Jian20} further investigated the convergence of the scalar curvature $R(t)$ under the same flow and established that
\begin{equation}\label{rlocalc0}
||R+\kappa||_{C^0(K)}\longrightarrow0
\end{equation}
uniformly on any compact subset $K\Subset X^o$. The uniform scalar curvature and diameter bound for the long-time solution of the K\"ahler-Ricci flow is a natural extension and analogue of Perelman's scalar curvature and diameter estimate \cite{SesT08} for the K\"ahler-Ricci flow of finite time extinction, i.e. the K\"ahler-Ricci flow on Fano manifolds with initial K\"ahler metric in the first Chern class \cite{So14,ZJS25+}. We highlight that the fact that the scalar curvature remains uniformly bounded along the normalized K\"ahler-Ricci flow \eqref{UKRF} plays a significant role in the aforementioned works.

In light of the earlier works cited above and the references therein, a recent result in \cite{LTZ26+} established that, for a compact K\"ahler manifold with semiample canonical bundle, the normalized K\"ahler-Ricci flow converges in the Gromov-Hausdorff topology to the metric completion of the generalized K\"ahler-Einstein metric on the canonical model, as conjectured by Song-Tian's analytic minimal model program.


\emph{There is a folklore question that the scalar curvature $R(t)$ of $\omega(t)$ should converge (pointwise) to negative Kodaira dimension of $X$.}

In this paper, we prove the following
\begin{theorem}\label{scalar-curvature-convergence}
Let $(X,\omega_0)$ be a compact K\"ahler manifold of complex dimension $n$ with semi-ample canonical line bundle, and let $\kappa\in\{0,1,2,\cdots,n\}$ be its Kodaira dimension. Assume that $\omega(t)$ is the long-time solution of the normalized K\"ahler-Ricci flow \emph{(\ref{UKRF})}, then we have
\begin{equation}
\lim_{t\to\infty} ||R(t)+\kappa||_{L^{\infty}}=0.
\end{equation}
Moreover, if $\kappa=n$, then there exists a constant $C=C(\omega_0,n)$ such that for all $t\geq0$,
\begin{equation}\label{rate-of-R}
-Ce^{-t}\leq R(t)+n\leq\frac{C}{1+t},\qquad\fint_X|R(t)+n|\omega(t)^n\leq Ce^{-t}.
\end{equation}
\end{theorem}

The proof of Theorem \ref{scalar-curvature-convergence} relies on the following uniform Sobolev inequality along the normalized K\"ahler-Ricci flow \eqref{UKRF} and the ultracontractivity property of Ricci flow (see Theorem \ref{ZhZl21} and Corollary \ref{heat-equation-NKRF}). The lower bound of the scalar curvature follows from the convergence \eqref{rlocalc0}, the ultracontractivity and the maximum principle. For the upper bound, the difficulty is that the evolution equation of $R$ contains the nonnegative term $|Ric|^2$, which may concentrate near the singular fibres of $f$. Near every fibre of $f$, including the singular ones, a local section of a power of $K_X$ pulled back from $X_{can}$ and the Ricci potential give a function $v$ with $\sqrt{-1}\partial\bar{\partial}v=-Ric(\omega(t))$, and the terms involving $|Ric|^2$ cancel in the evolution equation of $R+\kappa+|\nabla v|^2$. The sections are chosen by the estimates of Sz\'ekelyhidi \cite{Sz25+} and Lee-Tosatti-Zhang \cite{LTZ26+} on the canonical model, in such a way that the limit of $|\nabla v|^2$ away from the singular fibres is arbitrarily small. The ultracontractivity and a local maximum principle then give the upper bound across the singular fibres. The collapsing case $\kappa<n$ is treated in Section \ref{Proof of Theorem scalar-curvature-convergence}. In Section \ref{noncollapsing-case}, we consider the noncollapsing case $\kappa=n$. There, the evolution of the K\"ahler class yields integral estimates of order $e^{-t}$, which in turn yield the rate \eqref{rate-of-R}.

The Sobolev inequality is of evident importance because it contains a wealth of analytic and geometric information (e.g., volume noncollapsing and isoperimetric inequalities). Hence, achieving uniform constants in the Sobolev inequality is highly desirable. As a case in point, Perelman established a uniform Sobolev inequality for the K\"ahler-Ricci flow on Fano manifolds, see \cite{SesT08}. By applying Perelman's monotonicity formula, Zhang \cite{ZhQ16} proved a uniform Sobolev inequality for compact Ricci flows. We refer the reader to \cite[Section 6.2]{ZhQ16} for more information. Chan-Ma-Zhang further developed this idea, proving a uniform Sobolev inequality for ancient Ricci flows \cite{CMZ22} and a local Sobolev inequality for Ricci flows \cite{CMZ23}. Meanwhile, Li-Wang \cite{LuWa20}, using a distinct method, proved a uniform Sobolev inequality on shrinking gradient Ricci solitons.

In this paper, we also established a uniform Sobolev inequality for the normalized K\"ahler-Ricci flow \eqref{UKRF} on K\"ahler manifolds with semi-ample canonical line bundle.

\begin{theorem}\label{Sobolev-inequality-NKRF}
Let $(X,\omega_0)$ be a compact K\"ahler manifold of complex dimension $n$ with semi-ample canonical line bundle, then along the normalized K\"ahler-Ricci flow \eqref{UKRF}, we have the uniform entropy estimate
\begin{align}
\mu(g(t),\tau)\geq\log vol_{g(t)}(X,g(t))-C,\quad\forall~0<\tau\leq1,
\end{align}
and, if $n\geq2$, the uniform Sobolev inequality
\begin{align}
\Big(\fint_X|f|^{\frac{2n}{n-1}}\omega(t)^n\Big)^{\frac{n-1}{n}}\leq C_S\cdot\fint_X\big(|\nabla f|^2+f^2\big)\omega(t)^n,\quad f\in C^{\infty}(X,\mathbb{R}),
\end{align}
where $\fint$ denotes the average integration w.r.t. the measure $\omega(t)^n$. Here $C,~C_S$ are uniform constants depending only on $\omega_0,n$.
\end{theorem}

The proof of Theorem \ref{Sobolev-inequality-NKRF} draws on the recent results of Bamler \cite{Bam20} concerning estimates of Nash entropy, together with the work of Chan-Ma-Zhang \cite{CMZ23}. A more detailed discussion can be found in Section \ref{mu-entropy and Sobolev inequality along KRF}.

\section{Basic facts and estimates}
Let $(X,\omega_0)$ be a K\"ahler manifold with $K_X$ being semiample and $\omega(t)$ be the global solution of the normalized K\"ahler-Ricci flow (\ref{UKRF}). Throughout the paper, for a K\"ahler metric $\omega$ with associated Riemannian metric $g$, we denote by $R=tr_{\omega}Ric(\omega)$ the scalar curvature, by $\Delta=g^{i\bar j}\partial_i\partial_{\bar j}$ the Laplacian, and by $|\nabla v|^2=g^{i\bar j}\partial_iv\partial_{\bar j}v$ the squared norm of the gradient of a real-valued function $v$. The scalar curvature and Laplacian associated with the Riemannian metric $g$ are $2R$ and $2\Delta$, respectively. Now we will reduce the normalized K\"ahler-Ricci flow (\ref{UKRF}) to a parabolic Monge-Amp\`{e}re equation. Let $\mathcal{O}_{\mathbb{P}^N}(1)$ be the hyperplane bundle over $\mathbb{P}^N$ and $\omega_{FS}\in[\mathcal{O}_{\mathbb{P}^N}(1)]$ be a Fubini-Study metric on $\mathbb{P}^N.$ Then there exists $m>0$ such that
$$K_X=\frac{1}{m}f^*\mathcal{O}_{\mathbb{P}^N}(1).$$
Let $\chi=\frac{1}{m}f^*\omega_{FS}\in-2\pi c_1(X)$; it is a smooth nonnegative closed $(1,1)$-form on $X$. If $\kappa=0$, we take $\chi=0$. There also exists a smooth volume form $\Omega$ on $X$ such that
$$-Ric(\Omega)=\sqrt{-1}\partial\bar{\partial}\log\Omega=\chi.$$
Then K\"ahler class evolving along the normalized
K\"ahler-Ricci flow is given by
$$[\omega(t)]=e^{-t}[\omega_0]+(1-e^{-t})[\chi]$$
We define the reference metric
$$\omega_t=e^{-t}\omega_0+(1-e^{-t})\chi.$$
Then normalized K\"ahler-Ricci flow \eqref{UKRF} is equivalent to the following Monge-Amp\`{e}re equation.
\begin{equation}\label{npma}
  \begin{cases}
  \frac{\partial\varphi}{\partial t}&=\log\frac{e^{(n-\kappa)t}(\omega_t+\sqrt{-1}\partial\bar{\partial}\varphi)^n}{\Omega}-\varphi,\\
  \omega(t)&=\omega_t+\sqrt{-1}\partial\bar{\partial}\varphi>0,\\
  \varphi(0)&=0.
  \end{cases}
\end{equation}

Let $u=\frac{\partial\varphi}{\partial t}+\varphi$, which is referred to as the Ricci potential. A straightforward calculation yields the following lemma.

\begin{lemma}{\rm(\cite{ST16})}\label{Ricci-potential-lemma}
Under the normalized K\"ahler-Ricci flow \eqref{UKRF}, we have
\begin{align}
\frac{\partial}{\partial t}u&=\Delta u+tr_{\omega(t)}\chi-\kappa\\
Ric(\omega(t))&=-\sqrt{-1}\partial\bar{\partial}u-\chi
\end{align}
\end{lemma}

The parabolic gradient estimate was employed by Perelman to obtain a scalar curvature bound for the K\"ahler-Ricci flow on Fano manifolds \cite{SesT08}. This type of gradient estimate has found further application in the study of K\"ahler manifolds with positive Kodaira dimension \cite{ST07,ST12,ST16,Jian20}. We present below the Song-Tian $C^1$-estimate for Ricci potential \cite{ST16}, together with the diameter bound established by Jian-Song \cite{JS22} and the estimate of the volume form.

\begin{lemma}{\rm(\cite{ST16,JS22})}
There exists $C$ depending on $n$ and $\omega_0,$ such that on $X\times[0,+\infty),$
\begin{equation}\label{stc1}
|\varphi|+tr_{\omega(t)}\chi+|\frac{\partial\varphi}{\partial t}|+|\nabla u|+|\Delta u|+|R|\leq C,
\end{equation}
\begin{equation}\label{diameter-bound-NKRF}
diam(X,\omega(t))\le C,
\end{equation}
and
\begin{equation}\label{volumenoncollased}
C^{-1}e^{-(n-\kappa)t}\Omega\leq\omega(t)^n\leq Ce^{-(n-\kappa)t}\Omega.
\end{equation}
In particular,
\begin{equation}\label{volume-of-X}
C^{-1}e^{-(n-\kappa)t}\leq vol_{\omega(t)}(X)=\int_X\omega(t)^n\leq Ce^{-(n-\kappa)t}.
\end{equation}
Here $R$, $\nabla$ and $\Delta$ are taken with respect to $\omega(t)$. 
\end{lemma}

The following result, which combines Chern-Weil theory and Song-Tian's $C^1$-estimate for Ricci potential \eqref{stc1}, was originally proved by Tian-Zhang \cite[Lemma 3.4]{TZhaZ16}. 

\begin{lemma}{\rm(\cite[Lemma 3.4]{TZhaZ16})}
There exists $C=C(\omega_0,n)$ such that
\begin{equation}\label{L2Ricbd}
\int_X(|\nabla\nabla u|^2+|\nabla\bar{\nabla}u|^2+|Ric|^2+|Rm|^2)dvol_{g(t)}\leq C,~~for~all~t\geq0.
\end{equation}
\end{lemma}

In what follows, we state and prove an integral estimate for the Hessian of the Ricci potential, which essentially means that it is a critical quantity that yields an upper bound on the scalar curvature along the normalized K\"ahler-Ricci flow. The following Lemma was employed by Tian-Zhang to obtain a Kato bound for Ricci potential along the K\"ahler-Ricci flow on Fano manifolds \cite[Lemma 2.1, Lemma 2.3]{TiZh15}.

\begin{lemma}\label{Hessian-estimates-1}
Let $u$ be the Ricci potential. There is a uniform constant $C=C(\omega_0,n)$ such that, for $0<r\leq1$,
\begin{equation}\label{average-integral-Ricci-estimates}
\fint_{B(x,r,t)}\big(|Hess\,u|^2+|Ric|^2\big)\omega(t)^n
\leq C(1+r^{-2}).
\end{equation}
Here $|Hess\,u|^2=|\nabla\nabla u|^2+|\nabla\bar\nabla u|^2$, $B(x,r,t)$ is the metric ball centered at $x$ with radius $r$ with respect to $\omega(t)$, and $vol_{\omega(t)}B(x,r,t)$ is its volume with respect to $\omega(t)$. 
\end{lemma}
\begin{proof}
Fix $t$ and choose a Lipschitz cut-off function $\phi$ supported in $B(x,2r,t)$, equal to one on $B(x,r,t)$, with $0\leq\phi\leq1$ and $|\nabla\phi|\leq C/r$. Recall that
$$Ric=-\sqrt{-1}\partial\bar\partial u-\chi,\qquad
\Delta u=-R-tr_{\omega}\chi.$$
The Bochner formula gives
$$\Delta|\nabla u|^2=|Hess\,u|^2+Ric(\nabla u,\bar\nabla u)
+2Re\langle\nabla\Delta u,\nabla u\rangle.$$
So above equation integrating $\phi^2$ at $t$ yields
\begin{eqnarray*}
2\fint|Hess u|^2\phi^2dg_t&=&\fint\Delta|\nabla u|^2\phi^2-2\fint Ric(\nabla u,\bar{\nabla}u)\phi^2+2\fint Re\langle\nabla\Delta u,\nabla u\rangle\phi^2\\
&=&-4\fint\langle\nabla|\nabla u|,\nabla\phi\rangle|\nabla u|\phi dg_t+2\fint\langle\sqrt{-1}\partial\bar{\partial}u+\chi,\bar{\nabla}u\otimes\nabla u\rangle\phi^2\\
&-&2\fint(\Delta u)^2\phi^2-4\fint \Delta u\langle\nabla u,\nabla\phi\rangle\phi.
\end{eqnarray*}
Note that $|\nabla|\nabla u||\le |Hess u|$. An application of the Cauchy-Schwarz inequality, together with Song-Tian's estimates \eqref{stc1} for $|\nabla u|$, $|\Delta u|$ and $R$, yields
\begin{eqnarray*}
\fint|Hess u|^2\phi^2dg_t&=&\frac{C_1}{r^2}\fint_{B(x,r,t)}\big(|\nabla u|^2+|\nabla u|^4\big)+2\fint_{B(x,r,t)}(\Delta u)^2+\frac{C_1}{r^2}\fint|\Delta u||\nabla u|\\
&\le&C_2\cdot r^{-2}.
\end{eqnarray*}
The desired estimates \eqref{average-integral-Ricci-estimates} follows.
\end{proof}

We end this section with a local potential of $\chi$ adapted to the canonical metric, which will be used in Section \ref{Proof of Theorem scalar-curvature-convergence} and Section \ref{noncollapsing-case} to estimate the upper bound of the scalar curvature. Assume that $0<\kappa\leq n$. Let $\varphi_{can}$ be a potential of $\omega_{can}$ with respect to $\frac{1}{m}\omega_{FS}$, namely
$$\omega_{can}=\frac{1}{m}\omega_{FS}+\sqrt{-1}\partial\bar{\partial}\varphi_{can}\qquad\mbox{on}~X_{can}^o.$$
It is determined up to an additive constant, it is smooth on $X_{can}^o$ and it extends to a continuous function on $X_{can}$ \cite{ST12,LTZ26+}. Denote by $d_{FS}$ the distance function of $\omega_{FS}$ on $\mathbb{P}^N$ and, for $y\in X_{can}$ and $\rho>0$, put
$$B_{FS}(y,\rho)=\{z\in X_{can}|~d_{FS}(z,y)<\rho\}.$$

\begin{lemma}\label{local-canonical-potential}
There exist $k_0\in\mathbb{Z}^+$ and constants $b,C>0$ depending only on $\omega_0,n$ with the following property. For any positive multiple $k$ of $k_0$ and any $y\in X_{can}$, put $\rho=e^{-bk}$. Then there exists a real continuous function $\psi$ on $B_{FS}(y,4\rho)$ such that $f^*\psi$ is smooth on $f^{-1}(B_{FS}(y,4\rho))$,
\begin{equation}\label{local-potential-of-chi}
\sqrt{-1}\partial\bar{\partial}f^*\psi=\chi\qquad\mbox{on}~f^{-1}(B_{FS}(y,4\rho)),
\end{equation}
and
\begin{align}
|\partial(\varphi_{can}+\psi)|^2_{\omega_{can}}&\leq\frac{C}{k}\qquad~~\mbox{on}~B_{FS}(y,4\rho)\cap X_{can}^o,\label{local-canonical-gradient}\\
|\nabla f^*\psi|^2&\leq Ce^{Ck}\qquad\mbox{on}~f^{-1}(B_{FS}(y,4\rho))\times[0,+\infty),\label{local-reference-gradient}
\end{align}
where $\nabla$ is the gradient with respect to $\omega(t)$.
\end{lemma}
\begin{proof}
The proof will be divided into following three steps.

\textbf{Step 1: Sections of the polarization of $X_{can}$.} Let $L=\mathcal{O}_{\mathbb{P}^N}(1)|_{X_{can}}$ and let $h_{FS}$ be the Hermitian metric of $L$ induced by the Fubini-Study metric, whose curvature form is $\omega_{FS}$. Then
$$h=h_{FS}\cdot e^{-m\varphi_{can}}$$
is a continuous Hermitian metric of $L$ which is smooth on $X_{can}^o$, and its curvature form on $X_{can}^o$ is $m\omega_{can}$. For $k\in\mathbb{Z}^+$ and $\sigma\in H^0(X_{can},L^k)$ we put
$$\|\sigma\|_k^2=\int_{X_{can}^o}|\sigma|^2_{h^k}(mk\omega_{can})^{\kappa}.$$
We recall that $\omega_{can}$ has bounded local potentials and satisfies $$Ric(\omega_{can})=-\omega_{can}+\omega_{WP}\geq-\omega_{can}~~ \text{on}~~ X_{can}^o$$
and that the metric completion of $(X_{can}^o,\omega_{can})$ is homeomorphic to $X_{can}$ and is a non-collapsed $\mathrm{RCD}(-1,2\kappa)$-space \cite[Theorem 15]{Sz25+}, \cite[Theorem 2.1]{LTZ26+}. More precisely, $\omega_{can}$ satisfies the conditions (1), (2), (3) of \cite[Theorem 3]{Sz25+}. All these properties are preserved if $\omega_{can}$ is replaced by $m\omega_{can}$, with $\mathrm{RCD}(-1,2\kappa)$ replaced by $\mathrm{RCD}(-\frac{1}{m},2\kappa)$. We shall use the following three estimates, in which $\nabla$ denotes the Chern connection of the indicated Hermitian metric.

{\rm(a)} By the lower bound of the density of states function \cite[Theorem 3.2]{LTZ26+}, there exist $k_0\in\mathbb{Z}^+$ and $c_0>0$ such that for any positive multiple $k$ of $k_0$ and any $y\in X_{can}$,
$$\sup\big\{|\sigma(y)|^2_{h^k}~\big|~\sigma\in H^0(X_{can},L^k),~\|\sigma\|_k=1\big\}\geq c_0^2.$$
Since $H^0(X_{can},L^k)$ is finite dimensional, the supremum is attained. Hence there exists $\sigma\in H^0(X_{can},L^k)$ such that
\begin{equation}\label{peak-section}
\|\sigma\|_k=1,\qquad|\sigma(y)|_{h^k}\geq c_0.
\end{equation}

{\rm(b)} There exists $C$ such that for all $k\in\mathbb{Z}^+$ and $\sigma\in H^0(X_{can},L^k)$,
\begin{equation}\label{section-gradient-can}
\sup_{X_{can}^o}\Big(|\sigma|_{h^k}+|\nabla\sigma|_{h^k,mk\omega_{can}}\Big)\leq C\|\sigma\|_k.
\end{equation}
Indeed, let $q\in X_{can}$ and let $U$ be a Stein neighborhood of $q$ in $X_{can}$ on which $L$ admits a nowhere vanishing holomorphic section $e$. Write $|e|^2_h=e^{-w}$. After shrinking $U$, the function $w$ is bounded on $U$ and $\sqrt{-1}\partial\bar{\partial}w=m\omega_{can}$ on $U\cap X_{can}^o$. Let $U'\Subset U$ be a smaller neighborhood of $q$. The local estimate of Sz\'ekelyhidi \cite[Proposition 7]{Sz25+}, applied to the K\"ahler metric $m\omega_{can}=\sqrt{-1}\partial\bar{\partial}w$ and the Hermitian metric $e^{-w}$ of the trivial line bundle, gives a constant $C$ independent of $k$ such that
$$\sup_{U'\cap X_{can}^o}\Big(|\sigma|_{h^k}+|\nabla\sigma|_{h^k,mk\omega_{can}}\Big)\leq C\Big(\int_{U\cap X_{can}^o}|\sigma|^2_{h^k}(mk\omega_{can})^{\kappa}\Big)^{\frac{1}{2}}\leq C\|\sigma\|_k$$
for any holomorphic section $\sigma$ of $L^k$ over $U$. Here, as in \cite{Sz25+}, the estimate is first obtained on the complement of a divisor containing $X_{can}\setminus X_{can}^o$, and it holds on $U'\cap X_{can}^o$ by continuity. Covering $X_{can}$ by finitely many such $U'$ we get \eqref{section-gradient-can}.

{\rm(c)} By \cite[Lemma 3.4]{LTZ26+}, there exists $C$ such that for all $k\in\mathbb{Z}^+$ and $\sigma\in H^0(X_{can},L^k)$,
\begin{equation}\label{section-gradient-FS}
\sup_{X_{can}^{reg}}|\nabla\sigma|_{h_{FS}^k,\omega_{FS}}\leq e^{Ck}\|\sigma\|_k,
\end{equation}
where $X_{can}^{reg}$ is the regular part of $X_{can}$.

\textbf{Step 2: Lower bound of the section on $B_{FS}(y,4\rho)$.} Fix a positive multiple $k$ of $k_0$ and $y\in X_{can}$, and let $\sigma$ be the section in \eqref{peak-section}. Since $\varphi_{can}$ is bounded,
\begin{equation}\label{peak-section-FS}
|\sigma(y)|_{h_{FS}^k}=|\sigma(y)|_{h^k}\cdot e^{\frac{mk}{2}\varphi_{can}(y)}\geq c_0e^{-Ck}.
\end{equation}
By \eqref{section-gradient-FS} and $\big|\nabla|\sigma|_{h_{FS}^k}\big|\leq|\nabla\sigma|_{h_{FS}^k}$, the function $|\sigma|_{h_{FS}^k}$ is Lipschitz on $X_{can}^{reg}$, with Lipschitz constant $e^{Ck}$, with respect to the intrinsic distance of $\omega_{FS}|_{X_{can}}$. This intrinsic distance is bounded by $Cd_{FS}^{\alpha}$ for some $\alpha\in(0,1)$ by \cite[(3.44)]{LTZ26+}, and $\varphi_{can}$ is H\"older continuous with respect to it by \cite[Theorem 3.1]{LTZ26+}. Hence, decreasing $\alpha$ if necessary, we have for all $z\in X_{can}$
\begin{equation}\label{Holder-section-potential}
\Big||\sigma(z)|_{h_{FS}^k}-|\sigma(y)|_{h_{FS}^k}\Big|\leq Ce^{Ck}d_{FS}(z,y)^{\alpha},\qquad
|\varphi_{can}(z)-\varphi_{can}(y)|\leq Cd_{FS}(z,y)^{\alpha},
\end{equation}
where the first inequality is first obtained for $z,y\in X_{can}^{reg}$ and then extended to $X_{can}$ by continuity. By \eqref{peak-section-FS} and \eqref{Holder-section-potential} we can choose $b$ sufficiently large, depending only on $c_0,C,\alpha,m$, such that for $\rho=e^{-bk}$ and all $z\in B_{FS}(y,4\rho)$,
$$\Big||\sigma(z)|_{h_{FS}^k}-|\sigma(y)|_{h_{FS}^k}\Big|\leq\frac{1}{2}|\sigma(y)|_{h_{FS}^k},\qquad mk|\varphi_{can}(z)-\varphi_{can}(y)|\leq1.$$
Consequently, on $B_{FS}(y,4\rho)$,
\begin{equation}\label{section-lower-bound}
|\sigma|_{h_{FS}^k}\geq\frac{c_0}{2}e^{-Ck},\qquad|\sigma|_{h^k}=|\sigma|_{h_{FS}^k}\cdot e^{-\frac{mk}{2}\varphi_{can}}\geq\frac{c_0}{2\sqrt{e}}.
\end{equation}

\textbf{Step 3: The function $\psi$.} Define
$$\psi=-\frac{1}{mk}\log|\sigma|^2_{h_{FS}^k}\qquad\mbox{on}~B_{FS}(y,4\rho).$$
By \eqref{section-lower-bound}, $\psi$ is a real continuous function. Since $f^*L^k=K_X^{mk}$, the pull back $f^*\sigma$ is a nowhere vanishing holomorphic section of $K_X^{mk}$ over $f^{-1}(B_{FS}(y,4\rho))$, and $f^*h_{FS}$ is a smooth Hermitian metric of $f^*L$ whose curvature form is $f^*\omega_{FS}=m\chi$. Consequently,
$$
f^*\psi=-\frac{1}{mk}\log|f^*\sigma|^2_{f^*h_{FS}^k}
$$
is smooth and satisfies \eqref{local-potential-of-chi} on all of $f^{-1}(B_{FS}(y,4\rho))$, including the singular fibres of $f$.

On $B_{FS}(y,4\rho)\cap X_{can}^o$ we have $\varphi_{can}+\psi=-\frac{1}{mk}\log|\sigma|^2_{h^k}$. Then \eqref{section-gradient-can}, \eqref{peak-section} and \eqref{section-lower-bound} give
$$|\partial(\varphi_{can}+\psi)|^2_{\omega_{can}}\leq\frac{|\nabla\sigma|^2_{h^k,\omega_{can}}}{(mk)^2|\sigma|^2_{h^k}}=\frac{|\nabla\sigma|^2_{h^k,mk\omega_{can}}}{mk|\sigma|^2_{h^k}}\leq\frac{C}{k},$$
which is \eqref{local-canonical-gradient}. In the same way, \eqref{section-gradient-FS} and \eqref{section-lower-bound} give on $B_{FS}(y,4\rho)\cap X_{can}^{reg}$
$$|\partial\psi|^2_{\omega_{FS}}\leq\frac{|\nabla\sigma|^2_{h_{FS}^k,\omega_{FS}}}{(mk)^2|\sigma|^2_{h_{FS}^k}}\leq Ce^{Ck}.$$
Hence by \eqref{stc1}, we obtain $|\nabla f^*\psi|^2\leq Ce^{Ck}$ on $f^{-1}(B_{FS}(y,4\rho)\cap X_{can}^{reg})$ for all $t\geq0$. Since $f^{-1}(X_{can}^{reg})$ is the complement of a proper analytic subset of $X$ and $f^*\psi$ is smooth, the same bound holds on $f^{-1}(B_{FS}(y,4\rho))$. This proves \eqref{local-reference-gradient}.
\end{proof}



\section{Sobolev inequality along the K\"ahler-Ricci flow}
We know that Sobolev inequality contains various analytical and geometric information, including volume noncollapsing and isoperimetric inequality e.g. It is an important tool in studying elliptic and parabolic differential equations on manifolds.
\subsection{Bamler's estimates on Nash entropy}
In this subsection, we recall Bamler's estimates on Nash entropy \cite{Bam20}. Let $\tilde{g}(s),~s\in[0,1]$ be a solution to the Ricci flow
\begin{equation}\label{RF}
\frac{\partial}{\partial s}\tilde{g}(s)=-2\widetilde{Ric}(s)
\end{equation}
on compact manifold $M$ of real dimension $m$. Let $\tilde{\Delta}$ and $\widetilde{R}$ be the Laplacian and scalar curvature of $\tilde{g}$ respectively. Note that the evolution equation for the scalar curvature reads
$$\frac{\partial}{\partial s}\widetilde{R}=\tilde{\Delta}\widetilde{R}+2|\widetilde{Ric}|^2\geq\tilde{\Delta}\widetilde{R}+\frac{2}{m}\widetilde{R}^2.$$
Applying the maximum principle, we can deduce the following lower bound on the scalar curvature: if $\widetilde{R}(\cdot,t)\geq \widetilde{R}_{min},$ then for all $t\geq t_0,$ we have
\begin{equation}\label{scalarlb}
\widetilde{R}(\cdot,t)\geq\frac{m}{2}\frac{\widetilde{R}_{min}}{\frac{m}{2}-\widetilde{R}_{min}(t-t_0)}.
\end{equation}
In particular, if $t_0\in[0,T),$ then
\begin{equation}\label{scalarlb1}
\widetilde{R}(\cdot,t)\geq-\frac{m}{2(t-t_0)}.
\end{equation}
For any $\tau>0$ and let $d\nu=(4\pi\tau)^{-m/2}e^{-\widetilde{f}}d\widetilde{g}$ be a probability measure on $M$, then we define
$$\mathcal{N}[\widetilde{g},\widetilde{f},\tau]=\int_M\widetilde{f}d\nu-\frac{m}{2},\quad \mathcal{W}[\widetilde{g},\widetilde{f},\tau]=\int_M\big(\tau(|\nabla\widetilde{f}|^2+\widetilde{R})+\widetilde{f}-m\big) d\nu.$$
Consider a conjugate heat kernel measure
$$d\nu_{x_0,t_0}=(4\pi\tau)^{-m/2}e^{-\widetilde{f}}d\widetilde{g}=\widetilde{K}(x_0, t_0;\cdot,\cdot)d\widetilde{g}$$
based at some point $(x_0,t_0)\in M\times I$, where $\tau=t_0-t$.
\begin{definition}{\rm(\cite[Definition 5.1]{Bam20})}
The \textbf{pointed Nash entropy} at $(x_0,t_0)$ is defined as
$$\mathcal{N}_{x_0,t_0}(\tau):=\mathcal{N}[\widetilde{g}_{t_0-\tau}, \widetilde{f}_{t_0-\tau},\tau].$$
We set $\mathcal{N}_{x_0,t_0}(0):=0$. For $s<t_0$, $s\in I$, we also write
$$\mathcal{N}_s^*(x_0,t_0):=\mathcal{N}_{x_0,t_0}(t_0-s).$$
\end{definition}

For reader's convenience, we recall Bamler's Harnack inequality for the Nash entropy and Bamler's noninflating estimate for distance ball.
\begin{proposition}{\rm(\cite[Corollary 5.11]{Bam20})}\label{Bamler's Harnack inequality for the Nash entropy}
If $R(t^*)\geq R_{\min}, s<t^*\leq \min\{t_1,t_2\}$, and $s,t^*,t_1,t_2\in I $, then for any $x_1,x_2 \in M $, we have
\begin{eqnarray*}
\mathcal{N}_{x_1,t_1}(t_1-s)-\mathcal{N}_{x_2,t_2}(t_2-s)&\leq&\left(\frac{m}{2(t^*-s)}-R_{\min}\right)^{\frac{1}{2}} \operatorname{dist}_{W_1}^{\widetilde{g}_{t^*}}\big(\nu_{x_1,t_1}(t^*),\nu_{x_2,t_2}(t^*)\big)\\
&~&+\frac{m}{2} \log\left(\frac{t_2-s}{t^*-s}\right).
\end{eqnarray*}
\end{proposition}

\begin{proposition}{\rm(\cite[Theorem 8.1]{Bam20})}\label{Bamler's noninflating estimate for distance ball}
If $[t-r^2, t] \subset I$ and $R\geq R_{\min}$ on $M \times[t-r^2,t]$, then for any $1\leq A <\infty$, the following holds
\begin{align*}
|B(x, t, Ar)|_t \leq C(R_{\min} r^2) \exp(\mathcal{N}_{x,t}(r^2)) \exp(C_0 A^2) r^m.
\end{align*}
Here $C_0$ denotes constant a dimensional constant.
\end{proposition}

\subsection{Log Sobolev inequality and Sobolev inequality along K\"ahler-Ricci flow}\label{mu-entropy and Sobolev inequality along KRF}
In this subsection, we strengthen the monotonicity of Perelman's $\mathcal{W}$-entropy to a uniform Sobolev inequality along Ricci flow. We adopt the argument in \cite{Bam20,CMZ23} to show the uniform estimate of Sobolev constant under K\"ahler-Ricci flow (\ref{UKRF}). In this subsection, all constants $C$ in this section depend on $\omega_0,n$ unless otherwise specified.

\begin{theorem}{\rm(= Theorem \ref{Sobolev-inequality-NKRF})}\label{ZhZh22}
Under the assumption of {\rm Theorem \ref{Sobolev-inequality-NKRF}}, then along the normalized K\"ahler-Ricci flow \eqref{UKRF}, we have the uniform entropy estimate
\begin{align}\label{u-lowerbd}
\mu(g(t),\tau)\geq\log vol_{g(t)}(X,g(t))-C,\quad\forall~0<\tau\leq1,
\end{align}
and, if $n\geq2$, the uniform Sobolev inequality
\begin{align}\label{Soboleviq}
\Big(\fint_X|f|^{\frac{2n}{n-1}}\omega(t)^n\Big)^{\frac{n-1}{n}}\leq C_S\cdot\fint_X\big(|\nabla f|^2+f^2\big)\omega(t)^n,\quad f\in C^{\infty}(X,\mathbb{R}).
\end{align}
\end{theorem}
\begin{proof} For readers' convenience, we give the proof in the following steps. Throughout, the real dimension is $m=2n$ and $g(t)$ denotes the Riemannian metric of $\omega(t)$.

\textbf{Step 1: Adjusting Ricci flow.} Fix $t_0\geq0.$ Let $$\tilde{g}(s):=(1+2s)\cdot g(t_0+\log(1+2s)),~~0\leq s\leq1,$$
where $g(t_0)$ is the K\"ahler metric of $\omega(t_0).$ Then, $\tilde g(s)=(1+2s)g(t_0+\log(1+2s))$, with $1\leq1+2s\leq3$. In particular, $\tilde{g}(s)$ satisfies the Ricci flow
\begin{equation}\label{RF1}
\frac{\partial}{\partial s}\tilde{g}(s)=-2\widetilde{Ric}(s),~~\tilde{g}(0)=g(t_0).
\end{equation}
By \eqref{stc1} the real scalar curvature $\widetilde R=2(1+2s)^{-1}R(t_0+\log(1+2s))$ of $\tilde g$ satisfies $|\widetilde R|\le C$ on $X\times[0,1]$. By \eqref{diameter-bound-NKRF}, \eqref{volume-of-X}
\begin{equation}\label{tilde-g-bounds}
diam(X,\tilde g(s))\le C,\qquad C^{-1}vol_{g(t_0)}(X)\le vol_{\tilde g(s)}(X)\le C\,vol_{g(t_0)}(X),\qquad s\in[0,1].
\end{equation}
Let $\tau>0$ and $\eta\in C^{\infty}(X)$ be a minimizer of $\mu(\tilde{g}(1),\tau).$ By Rothaus \cite{Ro81}, we may assume that $\eta$ is positive. Consider the following conjugate heat equation
\begin{equation}\label{che}
-\frac{\partial}{\partial s}w=\tilde{\Delta}w-\widetilde{R},~~s\in[0,1]
\end{equation}
with initial $w(1)=\eta^2.$ Let $\tilde{K}(x,s;y,1)$ be the conjugate heat kernel. Write $$\tilde{K}(x,s;y,1)=(4\pi\tau)^{-n}e^{-\tilde{f}(x,s;y,1)}$$
where $\tau=\tau(s)=1-s$ and let
$$d\nu_{y,1;s}(x)=\tilde{K}(x,s;y,1)d\tilde{g}_s(x)$$
be associated heat kernel measure, then we define the Nash entropy
$$\mathcal{N}_{y,1}(s)=\int_X\tilde{f}(x,s;y,1)d\nu_{y,1;s}(x)-n.$$
Let $w(x,s):=(4\pi\tau)^{-n}e^{-f}$ and $d\bar{\nu}_s(x):=w(x,s)d_{\tilde{g}(s)}(x).$
We also introduce
$$\overline{\mathcal{N}}(s):=\int_Xf(x,s)d\bar{\nu}_s(x)-n,$$
$$\overline{\mathcal{W}}(s):=\int_X\big[\tau(\tilde{R}+|\nabla f|^2)+f-2n\big]d\bar{\nu}_s.$$
In particular, Bamler \cite[Proposition 5.2]{Bam20} or Chan-Ma-Zhang \cite[Lemma 3.2]{CMZ23} proved
$$\frac{d}{ds}\overline{\mathcal{W}}(s)\geq0,\qquad\frac{d}{ds}\big(\overline{\mathcal{N}}(s)\big)=-\overline{\mathcal{W}}(s).$$

\textbf{Step 2: Bounding $\mu$ in terms of Nash entropy.} Applying above estimates of $\overline{\mathcal{N}}(s)$ and $\overline{\mathcal{W}}(s),$ we can obtain that
$$(\tau+1)\cdot\overline{\mathcal{N}}(0)-\tau\cdot\overline{\mathcal{N}}(1)\leq\overline{\mathcal{W}}(1)=\mu(\tilde{g}(1),\tau).$$
Note that
$$\overline{\mathcal{N}}(1)=2n+\overline{\mathcal{W}}(1)-\tau\int_X(\tilde{R}+|\nabla f|^2)d\bar{\nu}_s\leq2n+\mu(\tilde{g}(1),\tau)+n\tau,$$
where we use $\tilde{R}(1)\geq-n$ under Ricci flow by (\ref{scalarlb1}).
Hence:
$$\overline{\mathcal{N}}(0)\leq\mu(\tilde{g}(1),\tau)+2n\tau.$$

By definition of $\overline{\mathcal{N}}(0)$,
\begin{align*}
\overline{\mathcal{N}}(0)~&=~\int_X f(x,0)w(x,s)d_{\tilde{g}(0)}(x)\\
~&=~-\int_Xw\log wd_{\tilde{g}(0)}(x)-n-n\log\big(4\pi(1+\tau)\big)\\
~&\geq~-\int_Xw\log wd_{\tilde{g}(0)}(x)-C(n),
\end{align*}
where $$w(x,0)=\int_Xw(y,1)\tilde{H}(x,0;y,1)d_{\tilde{g}(1)}(y)=\int_X\tilde{H}(x,0;y,1)d\bar{\nu}_1(y).$$
By Jensen inequality,
$$w(x,0)\log w(x,0)\leq\int_X\tilde{H}(x,0;y,1)\log\tilde{H}(x,0;y,1)d\bar{\nu}_1(y),$$
so,
\begin{align*}
\int_Xw\log wd_{\tilde{g}(0)}(x)~&\leq~\int_X\int_X\tilde{H}(x,0;y,1)\log\tilde{H}(x,0;y,1)d\bar{\nu}_1(y)d_{\tilde{g}(0)}(x)\\
~&=~\int_X\Big(\int_X\log\tilde{H}(x,0;y,1)d\bar{\nu}_0(x)\Big)d\bar{\nu}_1(y)\\
~&=~-\int_X\mathcal{N}_{y,1}(0)d\bar{\nu}_1(y)-n-n\log(4\pi).
\end{align*}
From definition of $\overline{\mathcal{N}}(0)$ and Jensen inequality, we have
$$\overline{\mathcal{N}}(0)\geq-\int_X\mathcal{N}_{y,1}(0)d\bar{\nu}_1(y)-C_1(n).$$
It remains to estimate $\mathcal{N}_{y,1}(0).$

\textbf{Step 3: Bounding $\mathcal{N}$ in terms of volume ration.} Applying Bamler's gradient estimate for Nash entropy $\mathcal{N}_{y,1}(0)$ in \cite[Corollary 5.11]{Bam20}, we get $\forall~x,y\in X$
$$\mathcal{N}_{y,1}(0)\leq\mathcal{N}_{x,1}(0)+\sqrt{2n}\cdot d_{\tilde{g}(1)}(x,y)\leq\mathcal{N}_{x,1}(0)+\sqrt{2n}\cdot diam(X,\tilde{g}(1)).$$
It follows that,
$$\overline{\mathcal{N}}(0)\geq-\inf_{y\in X}\mathcal{N}_{y,1}(0)-C_1(n)-\sqrt{2n}\cdot diam(X,\tilde{g}(1)).$$
Applying Bamler's upper volume bounds on distance balls with pointed Nash entropy in \cite[Theorem 8.1]{Bam20}, we obtain
$$vol(X,\tilde{g}(1))\leq C_2(n)\cdot e^{\mathcal{N}_{x,1}(0)}\cdot e^{C_2(n)(diam(X,\tilde{g}(0))^2}$$
for any $x\in X$. In particular,
$$\mathcal{N}_{x,1}(0)\geq\log vol(X,\tilde{g}(1))-\log C_2(n)-C_2(n)(diam(X,\tilde{g}(0)))^2$$
Arranging above estimates, we have
$$\mu(\tilde{g}(1),\tau)\geq\log vol(X,\tilde{g}(1))-\log C_2(n)-C_2(n)(diam(X,\tilde{g}(0)))^2-2n\tau.$$
Together with the estimate above and the fact that the metrics $\tilde{g}$ and $g$ are uniformly equivalent, we obtain the desired upper bound \eqref{u-lowerbd} for the scalar curvature along the normalized K\"ahler-Ricci flow.

\textbf{Step 4: Sobolev constant estimate.} Let $V=vol_{\omega(t)}(X)$. Writing $v^2=(4\pi\tau)^{-n}e^{-f}$, the estimate \eqref{u-lowerbd} says that for all $v\in C^\infty(X)$ with $\int_Xv^2\omega(t)^n=1$ and $\tau\in(0,1]$,
$$\int_Xv^2\log v^2\,\omega(t)^n\le\tau\int_X(4|\nabla v|_g^2+2Rv^2)\omega(t)^n-n\log(4\pi\tau)-2n-\log V+C.$$
Substituting $v=V^{-1/2}\phi$ with $\fint_X\phi^2\omega(t)^n=1$ and using $|R|\le C$, we arrive at the family of logarithmic Sobolev inequalities 
\begin{equation}\label{log-Sobolev-average}
\fint_X\phi^2\log\phi^2\,\omega(t)^n\le4\tau\fint_X|\nabla\phi|^2\omega(t)^n-n\log\tau+C,\qquad\fint_X\phi^2\omega(t)^n=1,~0<\tau\le1.
\end{equation}
Applying the classic arguments of Davies (see \cite[Theorem 6.2.1]{ZhQ16}) to the operator $\Delta$, which is self-adjoint with respect to $V^{-1}\omega(t)^n$, we deduce from the family of logarithmic Sobolev inequalities \eqref{log-Sobolev-average} the ultracontractivity of the heat semigroup associated with $g(t)$. 
In summary, we get our desired estimates \eqref{Soboleviq}.
\end{proof}

We present the following ultracontractivity property of positive solution to the heat equation, which may be of independent interest. In the sequel, this ultracontractivity property of Ricci flow will be applied frequently to prove the convergence of scalar curvature along the normalized K\"ahler-Ricci flow. A more detailed discussion can be found in Section \ref{Proof of Theorem scalar-curvature-convergence} and Section \ref{noncollapsing-case}.

\begin{theorem}\label{ZhZl21}
Let $q\ge1$ be any constant. Let $(M,g(t))$ be a solution of Ricci flow on a compact Riemannian manifold of real dimension $m$ on $M\times[0,t_0].$ Assume that
\begin{equation}\label{mu0lb}
\mu(g(0),\tau)\geq-A,~~\forall~0<\tau\leq t_0
\end{equation}
for some $A\geq0.$ Then for any nonnegative solution $u(t)$ of the heat equation $\frac{\partial}{\partial t}u=\Delta u$ coupled with $g(t)$,
\begin{equation}\label{ultra}
||u(t_0)||_{L^{\infty}}\leq C(m)\cdot e^{A+\lambda_{-}t_0}\cdot t_0^{-\frac{m}{2q}}\cdot||u(0)||_{L^q(g(0))}
\end{equation}
Here $R$ denotes the real scalar curvature, $\lambda_{-}=\max\{0,-\lambda\},$ $\lambda$ is the first eigenvalue of $-\frac{q-1}{q}\Delta+R$ at time $t=0.$ In particular, when $q=1$, $\lambda=\min_MR(\cdot,0)$ and \eqref{ultra} is an $L^1\to L^{\infty}$ estimate. The same estimate holds for nonnegative subsolutions of the heat equation.
\end{theorem}
\begin{proof}
Define $\tilde{q}(t)=\frac{qt_0}{t_0-t},~t\in[0,t_0).$ Put $$v=\frac{u^{\frac{\tilde{q}}{2}}}{||u^{\frac{\tilde{q}}{2}}||_{L^2}}$$ which has unit $L^2$-norm. Then
\begin{eqnarray*}
\frac{d}{dt}\log||u||_{L^{\tilde{q}}}&=&\frac{\tilde{q}'}{\tilde{q}^2}\int v^2\log v^2-\frac{\tilde{q}-1}{\tilde{q}^2}\int4|\nabla v|^2-\frac{1}{\tilde{q}}\int Rv^2\\
&=&\frac{\tilde{q}'}{\tilde{q}^2}\int v^2\log v^2-a(t)\int\big(4|\nabla v|^2+Rv^2\big)-b(t)\int\big(\frac{q-1}{q}|\nabla v|^2+Rv^2\big)
\end{eqnarray*}
where
$$a(t):=\frac{\tilde{q}(3q+1)-4q}{\tilde{q}^2(3q+1)}~~and~~b(t):=\frac{4q}{\tilde{q}^2(3q+1)}$$
are both positive constants. Notice that
$$\int\Big(\frac{q-1}{q}|\nabla v|^2+Rv^2\Big)\geq\lambda(t)$$
is the first eigenvalue of $R-\frac{q-1}{q}\Delta$ at time $t$. By the monotonicity of $\lambda(t)$ under the Ricci flow \cite[Theorem 1.1]{Li07}, we have
$$\int\Big(\frac{q-1}{q}|\nabla v|^2+Rv^2\Big)\geq\lambda(t)\geq-\lambda_-(0).$$
Put $\tau=aqt_0.$ Observe that, on the time interval $0\leq t\leq \frac{t_0}{2},$ we have $b(t)\leq1$ and
$$aq=1-\frac{t}{t_0}-\frac{4}{3q+1}\cdot(1-\frac{t}{t_0})^2\leq1-\frac{t}{t_0}.$$
So,
\begin{eqnarray*}
\frac{d}{dt}\log||u||_{L^{\tilde{q}}}
&\leq&\frac{\tilde{q}'}{\tilde{q}^2}\int v^2\log v^2-a(t)\int(4|\nabla v|^2+Rv^2)+\lambda_-(0)\\
&=&-\frac{1}{qt_0}\Big(aqt_0\int(4|\nabla v|^2+Rv^2)-\int v^2\log v^2\Big)+\lambda_-(0)\\
&\leq&-\frac{1}{qt_0}\Big(\mathcal{W}\big(g(t),v(t),aqt_0\big)+\frac{n}{2}\log\big(aqt_0\big)-C(n)\Big)+\lambda_-(0).
\end{eqnarray*}
Combining Perelman's monotonicity formula of $\mathcal{W}$-entropy with assumption (\ref{mu0lb}), we have
$$\mathcal{W}\big(g(t),v(t),aqt_0\big)\geq\mu\big(g(t),aqt_0\big)\geq\mu\big(g(0),aqt_0+t\big)\geq-A.$$
So we have, when $0\leq t\leq \frac{t_0}{2},$
\begin{eqnarray}
\nonumber\frac{d}{dt}\log||u||_{L^{\tilde{q}}}&\leq&-\frac{n}{2qt_0}\log t_0-\frac{n}{2qt_0}\log\Big(1-\frac{t}{t_0}-\frac{4}{3q+1}\cdot(1-\frac{t}{t_0})^2\Big)\\
&~&+\frac{C(n)+A}{qt_0}+\lambda_-(0).
\end{eqnarray}

On the other hand, when $\frac{t_0}{2}\leq t\leq t,$ we have
$R\geq-\frac{n}{t_0}$ by (\ref{scalarlb1}) and
\begin{eqnarray*}
\frac{d}{dt}\log||u||_{L^{\tilde{q}}}
&\leq&\frac{\tilde{q}'}{\tilde{q}^2}\int v^2\log v^2-\frac{\tilde{q}-1}{\tilde{q}^2}\int(4|\nabla v|^2+Rv^2)+\frac{n}{t_0}\\
&=&-\frac{1}{qt_0}\Big(\tilde{\tau}(t)\int(4|\nabla v|^2+Rv^2)-\int v^2\log v^2\Big)+\frac{n}{t_0}\\
&\leq&-\frac{1}{qt_0}\cdot\Big(\mathcal{W}\big(g(t),v(t),\tilde{\tau}\big)+\frac{n}{2}\log(4\pi\tilde{\tau})+n\Big)+\frac{n}{t_0},
\end{eqnarray*}
where
$$\tilde{\tau}:=qt_0\cdot\frac{\tilde{q}-1}{\tilde{q}^2}=\Big(\frac{q-1}{q}-\frac{q-2}{q}\frac{t}{t_0}-\frac{1}{q}(\frac{t}{t_0})^2\Big)\cdot t_0\leq t_0-t.$$
So we have, when $t_0\geq t\geq \frac{t_0}{2},$
\begin{eqnarray}
\nonumber\frac{d}{dt}\log||u||_{L^{\tilde{q}}}&\leq&-\frac{n}{2qt_0}\log t_0-\frac{n}{2qt_0}\log\Big(\frac{q-1}{q}-\frac{q-2}{q}\frac{t}{t_0}-\frac{1}{q}(\frac{t}{t_0})^2\Big)\\
&~&+\frac{A}{qt_0}+\frac{C(n)}{t_0}.
\end{eqnarray}
Then integrating from $t=0$ to $t=t_0$ gives
\begin{eqnarray*}
\log\frac{||u(t_0)||_{L^{\infty}}}{||u(0)||_{L^q(g(0))}}&\leq&-\frac{n}{2q}\log t_0+\frac{C(n)+A+\log q}{q}+\lambda_-(0)t_0+C(n)\\
&\leq&-\frac{n}{2q}\log t_0+A+\lambda_-(0)t_0+C(n).
\end{eqnarray*}
This gives our desired estimates.
\end{proof}

Combining the entropy estimate \eqref{u-lowerbd} with Theorem \ref{ZhZl21}, we obtain the following estimate for the heat equation along the normalized K\"ahler-Ricci flow \eqref{UKRF}, in which the volume of $(X,\omega(t))$ is allowed to tend to zero. It will be applied in Section \ref{Proof of Theorem scalar-curvature-convergence} and Section \ref{noncollapsing-case}.

\begin{corollary}\label{heat-equation-NKRF}
Let $(X,\omega_0)$ be a compact K\"ahler manifold of complex dimension $n$ with semi-ample canonical line bundle. Fix $t_0\geq0$ and let
$$\tilde{g}(s)=(1+2s)\cdot g(t_0+\log(1+2s)),~~0\leq s\leq1,$$
be the Ricci flow \eqref{RF1}. Then for any nonnegative solution $w(s)$ of the heat equation $\frac{\partial}{\partial s}w=\tilde{\Delta}w$ coupled with \eqref{RF1}, with continuous initial value $w(0)$, we have
\begin{equation}\label{ultra-normalized}
\|w(s)\|_{L^{\infty}}\leq C\cdot s^{-n}\cdot\fint_Xw(0)\omega(t_0)^n,\qquad0<s\leq1,
\end{equation}
where $C=C(\omega_0,n)$ is independent of $t_0$. Equivalently, if $w$ is a solution of the heat equation $\frac{\partial}{\partial t}w=\Delta w$ along the normalized K\"ahler-Ricci flow \eqref{UKRF} on $X\times[t_0,t_0+1]$, with continuous initial value $w(t_0)$, then
\begin{equation}\label{normalized-forward-heat}
\|w(t_0+h)\|_{L^{\infty}}\leq C\cdot h^{-n}\cdot\fint_X|w(t_0)|\omega(t_0)^n,\qquad0<h\leq1.
\end{equation}
\end{corollary}
\begin{proof}
By \eqref{u-lowerbd} and \eqref{volume-of-X}, we have
$$\mu(\tilde{g}(0),\tau)=\mu(g(t_0),\tau)\geq-A,~~\forall~0<\tau\leq1,\qquad A:=C-\log vol_{g(t_0)}(X)\geq0.$$
By \eqref{stc1}, the scalar curvature of $\tilde{g}(0)$ is bounded from below by $-C$. Applying Theorem \ref{ZhZl21} to the Ricci flow \eqref{RF1} of real dimension $2n$ on $X\times[0,s]$ with $q=1$, we get \eqref{ultra-normalized} as follows
$$\|w(s)\|_{L^{\infty}}\leq C(n)\cdot e^{A+Cs}\cdot s^{-n}\cdot\|w(0)\|_{L^1(\tilde{g}(0))}\leq C\cdot s^{-n}\cdot\fint_Xw(0)\omega(t_0)^n.$$

Recall that the Laplacian $\tilde{\Delta}$ of the Riemannian metric $\tilde{g}(s)$ equals $2(1+2s)^{-1}\Delta$, where $\Delta$ is the complex Laplacian of $\omega(t)$ at $t=t_0+\log(1+2s)$. Since $\frac{dt}{ds}=2(1+2s)^{-1}$, the function $s\mapsto w(t_0+\log(1+2s))$ is a solution of the heat equation coupled with \eqref{RF1}. Note that for $0<h\leq1$ we have $s=\frac{1}{2}(e^h-1)\in[\frac{h}{2},1]$. Then \eqref{normalized-forward-heat} follows from \eqref{ultra-normalized}.
\end{proof}


We end this section with the following technique lemma, which will be used to prove the upper bound of scalar curvature along the normalized K\"ahler-Ricci flow. A more detailed discussion can be found in Section \ref{Proof of Theorem scalar-curvature-convergence} and Section \ref{noncollapsing-case}.

\begin{lemma}\label{local-scalar-lemma}
Assume that $0<\kappa\leq n$. Let $k$ be a positive multiple of $k_0$, $y\in X_{can}$, and let $\rho=e^{-bk}$ and $\psi$ be as in Lemma {\rm\ref{local-canonical-potential}}. On $f^{-1}(B_{FS}(y,4\rho))\times[0,+\infty)$ put
\begin{equation}\label{definition-of-F}
v=u+f^*\psi,\qquad F=R+\kappa+|\nabla v|^2.
\end{equation}
Then
\begin{equation}\label{local-scalar-evolution}
\Big(\frac{\partial}{\partial t}-\Delta\Big)F=F-\kappa-|\nabla\nabla v|^2\leq F.
\end{equation}
Moreover, there exist constants $C$ and $C_1$ depending only on $\omega_0,n$ such that for any $a\geq0$, $0<h\leq1$ and $t\geq h$,
\begin{equation}\label{local-scalar-comparison}
\begin{split}
\sup_{f^{-1}(B_{FS}(y,\rho))}\big(R(t)+\kappa\big)\leq e^h\Big(a&+\frac{C\cdot h^{-n}}{vol_{\omega(t-h)}(X)}\int_{f^{-1}(B_{FS}(y,3\rho))}\big(F(t-h)-a\big)_+\omega(t-h)^n\\
&+C\cdot e^{C_1k}\cdot h\Big),
\end{split}
\end{equation}
where $(F-a)_+=\max\{F-a,0\}$.
\end{lemma}
\begin{proof}
\textbf{Step 1: The evolution equation of $F$.} By Lemma \ref{Ricci-potential-lemma} and \eqref{local-potential-of-chi}, on $f^{-1}(B_{FS}(y,4\rho))$ we have
$$\sqrt{-1}\partial\bar{\partial}v=\sqrt{-1}\partial\bar{\partial}u+\chi=-Ric(\omega(t)),\qquad\Delta v=-R$$
and
$$\frac{\partial}{\partial t}v=\Delta u+tr_{\omega(t)}\chi-\kappa=-R-\kappa,\qquad\Big(\frac{\partial}{\partial t}-\Delta\Big)v=-\kappa.$$
Along \eqref{UKRF} we have
$$\Delta|\nabla v|^2=|\nabla\nabla v|^2+|\nabla\bar{\nabla}v|^2+Ric(\nabla v,\bar{\nabla}v)+2Re\langle\nabla\Delta v,\nabla v\rangle.$$
Hence
\begin{eqnarray*}
\Big(\frac{\partial}{\partial t}-\Delta\Big)|\nabla v|^2&=&|\nabla v|^2-|\nabla\nabla v|^2-|\nabla\bar{\nabla}v|^2+2Re\Big\langle\nabla\Big(\frac{\partial}{\partial t}-\Delta\Big)v,\nabla v\Big\rangle\\
&=&|\nabla v|^2-|\nabla\nabla v|^2-|Ric|^2,
\end{eqnarray*}
where we use $|\nabla\bar{\nabla}v|^2=|Ric|^2$. Adding the evolution equation of $R$, we get \eqref{local-scalar-evolution}.

\textbf{Step 2: Cut-off functions.} By \eqref{stc1} and \eqref{local-reference-gradient}, we have
\begin{equation}\label{bound-of-F}
|F|\leq|R|+\kappa+2|\nabla u|^2+2|\nabla f^*\psi|^2\leq A:=Ce^{Ck}\qquad\mbox{on}~f^{-1}(B_{FS}(y,4\rho))\times[0,+\infty).
\end{equation}
Since $\rho=e^{-bk}$ is small, the function $d_{FS}(y,\cdot)^2$ is smooth on the ball of radius $4\rho$ centered at $y$ in $\mathbb{P}^N$. Composing it with fixed functions of one variable, we get smooth functions $\theta,\eta$ on $\mathbb{P}^N$ with values in $[0,1]$ such that
$$\theta=0~\mbox{on}~\{d_{FS}(y,\cdot)\leq\rho\},\qquad\theta=1~\mbox{on}~\{d_{FS}(y,\cdot)\geq\tfrac{3}{2}\rho\},$$
$$\eta=1~\mbox{on}~\{d_{FS}(y,\cdot)\leq2\rho\},\qquad\eta=0~\mbox{on}~\{d_{FS}(y,\cdot)\geq\tfrac{5}{2}\rho\},$$
and $$-C\rho^{-2}\omega_{FS}\leq\sqrt{-1}\partial\bar{\partial}\theta\leq C\rho^{-2}\omega_{FS}.$$
By \eqref{stc1}, we obtain
\begin{equation}\label{Laplacian-of-cutoff}
|\Delta\theta|\leq C\rho^{-2}\cdot tr_{\omega(t)}f^*\omega_{FS}=C\rho^{-2}\cdot m\cdot tr_{\omega(t)}\chi\leq C\rho^{-2}=Ce^{2bk}.
\end{equation}

\textbf{Step 3: Apply the maximum principle.} Fix $a\geq0$, $0<h\leq1$ and $t\geq h$, and put $t_0=t-h$. Let $w$ be the solution of the heat equation $\frac{\partial}{\partial\tau}w=\Delta w$ along \eqref{UKRF} on $X\times[t_0,t]$ with initial value
$$w(t_0)=\eta\cdot\big(F(t_0)-a\big)_+,$$
which is extended by zero outside $f^{-1}(B_{FS}(y,3\rho))$ and is a nonnegative continuous function on $X$. Let $A$ be the constant in \eqref{bound-of-F} and let $A'$ be the constant on the right hand side of \eqref{Laplacian-of-cutoff}. On $f^{-1}(\overline{B_{FS}(y,2\rho)})\times[t_0,t]$, define
$$G(\tau)=e^{\tau-t_0}\Big(a+w(\tau)+A\cdot\theta+A\cdot A'\cdot(\tau-t_0)\Big).$$
Then
$$\Big(\frac{\partial}{\partial\tau}-\Delta\Big)G=G+e^{\tau-t_0}\cdot A\cdot(A'-\Delta\theta)\geq G.$$
At $\tau=t_0$, since $\eta=1$, we obtain
$$G(t_0)=a+\big(F(t_0)-a\big)_++A\cdot\theta\geq F(t_0)\qquad\mbox{on}~f^{-1}(\overline{B_{FS}(y,2\rho)}).$$
On $f^{-1}(\{d_{FS}(y,\cdot)=2\rho\})$, which contains the boundary of $f^{-1}(B_{FS}(y,2\rho))$ in $X$, one has $\theta=1$ and $w\geq0$. Hence $G\geq A\geq F$ for all $\tau\in[t_0,t]$. By \eqref{local-scalar-evolution}, we have
$$\Big(\frac{\partial}{\partial\tau}-\Delta\Big)\big(e^{-(\tau-t_0)}(F-G)\big)\leq0.$$
The maximum principle therefore yields $F\leq G$ on $f^{-1}(B_{FS}(y,2\rho))\times[t_0,t]$. On $f^{-1}(B_{FS}(y,\rho))$, we have $\theta=0$. Since $R+\kappa\leq F$, we obtain at $\tau=t$
$$\sup_{f^{-1}(B_{FS}(y,\rho))}\big(R(t)+\kappa\big)\leq e^h\Big(a+\|w(t)\|_{L^{\infty}}+A\cdot A'\cdot h\Big).$$
By \eqref{normalized-forward-heat} in Corollary \ref{heat-equation-NKRF}, we obtain
$$\|w(t)\|_{L^{\infty}}\leq C\cdot h^{-n}\cdot\fint_Xw(t_0)\omega(t_0)^n\leq\frac{C\cdot h^{-n}}{vol_{\omega(t_0)}(X)}\int_{f^{-1}(B_{FS}(y,3\rho))}\big(F(t_0)-a\big)_+\omega(t_0)^n.$$
Since $A\cdot A'\leq Ce^{C_1k}$, this proves \eqref{local-scalar-comparison}.
\end{proof}

\section{The collapsing case: proof of Theorem \ref{scalar-curvature-convergence} for $\kappa<n$}\label{Proof of Theorem scalar-curvature-convergence}

In this section, all constants $C$ depend on $\omega_0,n$ unless otherwise specified.

\subsection{Lower bound of scalar curvature}\label{lower-bound-section}
We first state a useful lemma in this subsection, one that is often invoked in the proof of the main theorem.

\begin{lemma}\label{lpto0}
Under the normalized K\"ahler-Ricci flow {\rm\eqref{UKRF}}, we have
\begin{equation}
\fint_X|R(t)+\kappa|^p\omega(t)^n\longrightarrow0~~as~t\longrightarrow\infty
\end{equation}
for any $p\geq1.$
\end{lemma}
\begin{proof}
For any $\varepsilon>0,$ we can choose a compact subset $K\subset X^o$ such that $vol_{\omega_0}(X\setminus K)\leq\varepsilon.$ Then for any $t>0,$ by \eqref{volumenoncollased},
$$e^{(n-\kappa)t}\cdot vol_{\omega(t)}(X\setminus K)\leq C\cdot vol_{\omega_0}(X\setminus K)\leq C\cdot\varepsilon.$$
Hence, combining above inequality with (\ref{rlocalc0}), \eqref{stc1} and \eqref{volume-of-X} yields
\begin{align*}
\fint_X|R(t)+\kappa|^p\omega(t)^n~&\leq~\frac{1}{vol_{\omega(t)}(X)}\int_{X\setminus K}|R(t)+\kappa|^p\omega(t)^n+||R(t)+\kappa||^p_{C^0(K)}\\
~&\leq~C\cdot e^{(n-\kappa)t}vol_{\omega(t)}(X\setminus K)+||R(t)+\kappa||^p_{C^0(K)}\\
~&\leq~C\cdot\varepsilon+||R(t)+\kappa||^p_{C^0(K)}.
\end{align*}
It follows that $$\fint_X|R(t)+\kappa|^p\omega(t)^n\leq C\cdot\varepsilon$$ whenever $t$ is sufficiently large. The desired estimate follows.
\end{proof}

We will use above lemma with general ultracontractivity property of Ricci flow to obtain lower bound of scalar curvature along the normalized K\"ahler-Ricci flow \eqref{UKRF}.

\begin{proposition}\label{slowerbd}
Under the normalized K\"ahler-Ricci flow {\rm\eqref{UKRF}}, we have that
\begin{equation}\label{s-inf=-k}
\inf_XR(t)\longrightarrow-\kappa,~~as~t\longrightarrow\infty.
\end{equation}
\end{proposition}
\begin{proof}
Since $\inf_X(R(t)+\kappa)\leq\fint_X(R(t)+\kappa)\omega(t)^n\to0$ by Lemma \ref{lpto0}, it suffices to show that $$\liminf_{t\to\infty}\inf_XR(t)\geq-\kappa.$$

\textbf{Step 1: Modify the Ricci flow.}
Fix $t_0\geq0.$ Let $$\tilde{g}(s):=(1+2s)\cdot g(t_0+\log(1+2s)),~~0\leq s\leq1.$$ Then $\tilde{g}(s)$ satisfies the Ricci flow (\ref{RF1}) with initial $g(t_0)$ (is the K\"ahler metric of $\omega(t_0)$). The scalar curvature $\tilde{R}$ of the Riemannian metric $\tilde{g}(s)$ is given by
$$\tilde{R}(s)=\frac{2R(t_0+\log(1+2s))}{1+2s}.$$
Recall the scalar curvature evolution
$$\frac{\partial}{\partial s}\tilde{R}=\tilde{\Delta}\tilde{R}+2|\widetilde{Ric}|^2\geq\tilde{\Delta}\tilde{R}.$$
Let $w(s)$ be a solution to the heat equation coupled with Ricci flow (\ref{RF1}) with initial value
$$w(0)=\big(\tilde{R}(0)+2\kappa\big)_{-}(:=\max\{0,-\tilde{R}(0)-2\kappa\})=2(R(t_0)+\kappa)_{-}.$$

\textbf{Step 2: Apply the ultracontractivity.}
By \eqref{ultra-normalized} in Corollary \ref{heat-equation-NKRF}, for any $s\in(0,1]$,
$$||w(s)||_{L^{\infty}}\leq C_1\cdot s^{-n}\cdot\fint_X|R(t_0)+\kappa|\omega(t_0)^n,$$
which tends to 0 as $t_0\to\infty$ for any fixed $s$, by Lemma \ref{lpto0}.

\textbf{Step 3: Apply the maximum principle.}
Note that $$\tilde{R}(0)+w(0)\geq-2\kappa,$$ then the maximum principle to $\tilde{R}+w$ gives
$$\tilde{R}(s)+w(s)\geq-2\kappa,~~\forall~~s\in[0,1].$$
Hence $$\tilde{R}(s)\geq-2\kappa-||w(s)||_{L^{\infty}},~~\forall~~s\in[0,1].$$
Return to the normalized K\"ahler-Ricci flow {\rm\eqref{UKRF}}, at time $t=t_0+\log(1+2s),$
\begin{eqnarray*}
R(t)&\geq&-(1+2s)\Big(\kappa+C_1\cdot s^{-n}\cdot\fint_X|R(t_0)+\kappa|\omega(t_0)^n\Big)\\
&\geq&-\kappa+O(s)~~as~t_0\to\infty,
\end{eqnarray*}
where we choose $$s=\Big(\fint_X|R(t_0)+\kappa|\omega(t_0)^n\Big)^{\frac{1}{2n}}\to0~as~t_0\to\infty.$$
Hence, we get the desired the estimates \eqref{s-inf=-k}. This completes the proof.
\end{proof}

\subsection{Upper bound of scalar curvature}\label{upper-bound-section}
In this subsection, we prove that the upper bound of the scalar curvature tends to the negative Kodaira dimension of $X$ along the normalized K\"ahler-Ricci flow {\rm\eqref{UKRF}}.
\begin{proposition}\label{upper-bd-scalar}
Suppose that $0<\kappa<n$. Under the normalized K\"ahler-Ricci flow {\rm\eqref{UKRF}}, we have
\begin{equation}\label{sup-R-limit}
\sup_XR(t)\longrightarrow-\kappa,~~as~t\longrightarrow\infty.
\end{equation}
\end{proposition}
\begin{proof}
Since $\sup_X(R(t)+\kappa)\geq\fint_X(R(t)+\kappa)\omega(t)^n\to0$ by Lemma \ref{lpto0}, it suffices to show that
$$\limsup_{t\to\infty}\sup_XR(t)\leq-\kappa.$$
Fix a positive multiple $k$ of $k_0$ and $y\in X_{can}$, and let $\rho$, $\psi$, $v$ and $F$ be as in Lemma \ref{local-canonical-potential} and Lemma \ref{local-scalar-lemma}. Denote by $C_0$ the constant in \eqref{local-canonical-gradient}.

\textbf{Step 1: The limit of $F$ on $X^o$.} We claim that
\begin{equation}\label{limit-of-F}
\limsup_{t\to\infty}\sup_KF(t)\leq\frac{C_0}{k}, ~~\text{any~ compact ~subset} ~K \subset f^{-1}(B_{FS}(y,4\rho))\cap X^o.
\end{equation}
Note that on $f^{-1}(B_{FS}(y,4\rho))\cap X^o$, the following holds:
\begin{equation}\label{decomposition-of-v}
\partial v=\partial(u-f^*\varphi_{can})+f^*\partial(\varphi_{can}+\psi).
\end{equation}
By \cite[Proposition 4.3]{Jian20}, we have
$$\sup_K|\nabla(u-f^*\varphi_{can})|\longrightarrow0,~~as~t\longrightarrow\infty.$$
By \cite{TWY18} (see \cite[Theorem 2.3]{Jian20} and \cite[(2.7)]{LTZ26+}), there exists a closed real $(1,1)$-form $\omega_{SRF}$ on $X^o$, whose restriction to each fibre $X_y$, $y\in X_{can}^o$, is a Ricci flat K\"ahler metric, such that
$$\big\|\omega(t)-(f^*\omega_{can}+e^{-t}\omega_{SRF})\big\|_{C^0(K,\omega(t))}\longrightarrow0,~~as~t\longrightarrow\infty.$$
Since $\omega_{SRF}$ is positive along the fibres of $f$ and $K$ is compact, we have $$\omega_{SRF}\geq-C_Kf^*\omega_{can}~~~\text{on}~K.$$ It follows that there exists $\varepsilon_K(t)\to0$ such that
$$f^*\omega_{can}\leq\big(1+\varepsilon_K(t)\big)\omega(t)\qquad\mbox{on}~K.$$
Applying this with \eqref{local-canonical-gradient}, we get
$$\limsup_{t\to\infty}\sup_K|f^*\partial(\varphi_{can}+\psi)|^2_{\omega(t)}\leq\frac{C_0}{k}.$$
The claim \eqref{limit-of-F} now follows from \eqref{rlocalc0}, \eqref{decomposition-of-v}.

\textbf{Step 2: Integral estimate of $F$.} We claim that
\begin{equation}\label{integral-of-F}
\frac{1}{vol_{\omega(t)}(X)}\int_{f^{-1}(B_{FS}(y,3\rho))}\Big(F(t)-\frac{C_0}{k}\Big)_+\omega(t)^n\longrightarrow0,~~as~t\longrightarrow\infty.
\end{equation}
Indeed, by \eqref{volumenoncollased} and \eqref{volume-of-X} we have $\omega(t)^n\leq C\cdot vol_{\omega(t)}(X)\cdot\Omega$. By \eqref{bound-of-F} the integrand is bounded by $Ce^{Ck}$ independently of $t$. Since $X\setminus X^o$ is a proper analytic subset of $X$, it has measure zero with respect to $\Omega$.  Moreover, by \eqref{limit-of-F}, the integrand converges pointwise to zero on $f^{-1}(B_{FS}(y,3\rho))\cap X^o$. The claim follows from the dominated convergence theorem.

\textbf{Step 3: Upper bound of scalar curvature.} Applying \eqref{local-scalar-comparison} with $a=\frac{C_0}{k}$ and letting $t\to\infty$ for fixed $h\in(0,1]$, we get from \eqref{integral-of-F}
$$\limsup_{t\to\infty}\sup_{f^{-1}(B_{FS}(y,\rho))}\big(R(t)+\kappa\big)\leq e^h\Big(\frac{C_0}{k}+C\cdot e^{C_1k}\cdot h\Big).$$
Hence
$$\limsup_{t\to\infty}\sup_X\big(R(t)+\kappa\big)\leq e^h\Big(\frac{C_0}{k}+C\cdot e^{C_1k}\cdot h\Big).$$
Letting $h\to0$ first and then $k\to\infty$, we get the desired upper bound for the scalar curvature along the normalized K\"ahler-Ricci flow. This completes the proof.
\end{proof}

{\bf Proof of Theorem \ref{scalar-curvature-convergence} for $\kappa<n$}: If $0<\kappa<n$, the assertion follows from Proposition \ref{slowerbd} and Proposition \ref{upper-bd-scalar}.

If $\kappa=0$, then $c_1(X)=0$ in $H^2(X,\mathbb{R})$. By Cao \cite{Cao85}, the solution $\tilde{\omega}(s)$ of the K\"ahler-Ricci flow \eqref{KRF} converges smoothly and exponentially fast to the Ricci flat K\"ahler metric in $[\omega_0]$ as $s\to\infty$ (see also \cite{SW13-1,TosZh15}). Hence there exist $C,c>0$ such that the scalar curvature of $\tilde{\omega}(s)$ is bounded by $Ce^{-cs}$. Since $\omega(t)=e^{-t}\tilde{\omega}(e^t-1)$, we have
$$\|R(t)\|_{L^{\infty}}=e^t\cdot\|R(\tilde{\omega}(e^t-1))\|_{L^{\infty}}\leq C\cdot e^t\cdot\exp\big(-c(e^t-1)\big)\longrightarrow0,~~as~t\longrightarrow\infty.$$
This completes the proof. \qed

\section{The noncollapsing case: proof of Theorem \ref{scalar-curvature-convergence} for $\kappa=n$}\label{noncollapsing-case}

In this section we assume $\kappa=n$, and all constants $C$ depend on $\omega_0,n$ unless otherwise specified. In this case, it follows from \cite{TZh06,Tsu88} that, as $t\to\infty$,
$$
\frac{\partial\varphi}{\partial t}\longrightarrow0,
\qquad
\omega(t)\longrightarrow f^*\omega_{can}
\quad\text{in }C^{\infty}_{loc}(X^o).
$$
Moreover, $\varphi(t)$ converges in $C^{\infty}_{loc}(X^o)$ to a bounded function $\varphi_{\infty}$ satisfying
$$(\chi+\sqrt{-1}\partial\bar{\partial}\varphi_{\infty})^n=e^{\varphi_{\infty}}\Omega,\qquad\chi+\sqrt{-1}\partial\bar{\partial}\varphi_{\infty}=f^*\omega_{can}\qquad\mbox{on}~X^o.$$
Since $\varphi_{\infty}-f^*\varphi_{can}$ is a bounded pluriharmonic function on $X^o$, it extends to a pluriharmonic function on the compact manifold $X$. Hence it is a constant. Throughout this section we fix the additive constant of $\varphi_{can}$ such that $\varphi_{\infty}=f^*\varphi_{can}$ on $X^o$. In particular,
\begin{equation}\label{limit-of-u}
u(t)=\frac{\partial\varphi}{\partial t}+\varphi\longrightarrow f^*\varphi_{can},\qquad\varphi(t)\longrightarrow f^*\varphi_{can}
\end{equation}
pointwise on $X^o$ as $t\to\infty$.

\subsection{Integral estimates}
In the noncollapsing case the evolution of the K\"ahler class gives the rate of convergence of the scalar curvature in the integral sense.

\begin{proposition}\label{noncollapsing-integrals}
For all $t\geq0$, we have
\begin{equation}\label{noncollapsing-lower}
R(t)+n\geq-Ce^{-t}
\end{equation}
and
\begin{equation}\label{noncollapsing-energy}
\fint_X\Big(|R(t)+n|+|Ric(\omega(t))+\omega(t)|^2\Big)\omega(t)^n\leq Ce^{-t}.
\end{equation}
Moreover,
\begin{equation}\label{noncollapsing-average}
\fint_X(R(t)+n)\omega(t)^n=n\cdot\frac{\int_X(\omega_0-\chi)\wedge\chi^{n-1}}{\int_X\chi^n}\cdot e^{-t}+O(e^{-2t}).
\end{equation}
\end{proposition}
\begin{proof}
Note that
\begin{equation}\label{einstein-scalar-evolution}
\Big(\frac{\partial}{\partial t}-\Delta\Big)(R+n)=-(R+n)+|Ric+\omega|^2\geq-(R+n).
\end{equation}
The maximum principle gives the required estimate \eqref{noncollapsing-lower} as follows $$R(t)+n\geq-e^{-t}\cdot\|(R(0)+n)_{-}\|_{L^{\infty}}.$$
Recall that $$[\omega(t)]=e^{-t}[\omega_0]+(1-e^{-t})[\chi]~~ \text{ and}~~[Ric(\omega(t))]=-[\chi].$$
Hence,
$$\int_X(R(t)+n)\omega(t)^n=n\int_X\big(Ric(\omega(t))+\omega(t)\big)\wedge\omega(t)^{n-1}=ne^{-t}\int_X(\omega_0-\chi)\wedge\big(\chi+e^{-t}(\omega_0-\chi)\big)^{n-1}.$$
Dividing by $vol_{\omega(t)}(X)=\int_X(\chi+e^{-t}(\omega_0-\chi))^n$ we get \eqref{noncollapsing-average}. Together with \eqref{noncollapsing-lower}, this gives
$$\fint_X|R(t)+n|\omega(t)^n \leq Ce^{-t}.$$
Note that for any real $(1,1)$-form $\beta$ we have $$\big((tr_{\omega}\beta)^2-|\beta|^2_{\omega}\big)\omega^n=n(n-1)\beta^2\wedge\omega^{n-2}.$$
Applying this to $\beta=Ric(\omega(t))+\omega(t)$, we get
\begin{eqnarray*}
\int_X|Ric+\omega|^2\omega^n&=&\int_X(R+n)^2\omega^n-n(n-1)e^{-2t}\int_X(\omega_0-\chi)^2\wedge\omega(t)^{n-2}\\
&\leq&C\int_X|R+n|\omega^n+Ce^{-2t}\leq Ce^{-t},
\end{eqnarray*}
where we use \eqref{stc1}. This proves \eqref{noncollapsing-energy}.
\end{proof}


\begin{lemma}\label{canonical-energy-comparison}
For all $t\geq0$, we have
\begin{equation}\label{canonical-energy}
\frac{1}{vol_{\omega(t)}(X)}\int_{X^o}\Big(|\nabla(u(t)-f^*\varphi_{can})|^2+\big|tr_{\omega(t)}f^*\omega_{can}-n\big|\Big)\omega(t)^n\leq Ce^{-t},
\end{equation}
where $\nabla$ is the gradient with respect to $\omega(t)$.
\end{lemma}
\begin{proof}
\textbf{Step 1: Comparison of $\omega(t)$ and $\omega(T)$.} Let $T\geq t\geq0$. Since $\kappa=n$, by Lemma \ref{Ricci-potential-lemma} and \eqref{npma}, we have
$$u(s)=\log\frac{\omega(s)^n}{\Omega},~~\frac{\partial}{\partial s}u=-(R(s)+n).$$
Since $\omega(s)^n$ are uniformly equivalent to $\Omega$, \eqref{noncollapsing-energy} gives
\begin{equation}\label{u-difference}
\int_X|u(T)-u(t)|\omega(t)^n\leq C\int_t^T\int_X|R(s)+n|\omega(s)^nds\leq Ce^{-t}.
\end{equation}
By the arithmetic-geometric mean inequality,
$$tr_{\omega(t)}\omega(T)\geq n\Big(\frac{\omega(T)^n}{\omega(t)^n}\Big)^{\frac{1}{n}}=n\exp\Big(\frac{u(T)-u(t)}{n}\Big).$$
Since $u$ is uniformly bounded by \eqref{stc1}, this estimate together with \eqref{u-difference} yields
$$\int_X\big(n-tr_{\omega(t)}\omega(T)\big)\omega(t)^n\leq C\int_X|u(T)-u(t)|\omega(t)^n\leq Ce^{-t}.$$
On the other hand,
\begin{eqnarray*}
\int_X\big(tr_{\omega(t)}\omega(T)-n\big)\omega(t)^n&=&n\int_X\big(\omega(T)-\omega(t)\big)\wedge\omega(t)^{n-1}\\
&=&n(e^{-T}-e^{-t})\int_X(\omega_0-\chi)\wedge\omega(t)^{n-1}\le Ce^{-t}.
\end{eqnarray*}
Therefore
\begin{equation}\label{smooth-trace-comparison}
\int_X\big|tr_{\omega(t)}\omega(T)-n\big|\omega(t)^n\leq Ce^{-t},\qquad\forall~T\geq t.
\end{equation}

\textbf{Step 2: Gradient estimate.} Recall that $$\omega(T)=\omega_T+\sqrt{-1}\partial\bar{\partial}\varphi(T),~~~\Delta u(t)=-R(t)-tr_{\omega(t)}\chi,$$
where $\omega_T=e^{-T}\omega_0+(1-e^{-T})\chi$, and $\Delta$ is the Laplacian of $\omega(t)$. Hence
\begin{eqnarray*}
\Delta\big(u(t)-\varphi(T)\big)&=&-R(t)-tr_{\omega(t)}\chi-tr_{\omega(t)}\omega(T)+tr_{\omega(t)}\omega_T\\
&=&-(R(t)+n)-\big(tr_{\omega(t)}\omega(T)-n\big)+e^{-T}tr_{\omega(t)}(\omega_0-\chi).
\end{eqnarray*}
Since $\omega_0>0$ and $\chi\geq0$,
$$\int_X\big|tr_{\omega(t)}(\omega_0-\chi)\big|\omega(t)^n\leq n\int_X(\omega_0+\chi)\wedge\omega(t)^{n-1}\leq C.$$
Combining this with \eqref{noncollapsing-energy} and \eqref{smooth-trace-comparison}, we get
$$\int_X\big|\Delta\big(u(t)-\varphi(T)\big)\big|\omega(t)^n\leq Ce^{-t}+Ce^{-T}\leq Ce^{-t}.$$
Since $|u(t)-\varphi(T)|\leq C$ by \eqref{stc1}, an integration by parts on $X$ gives
\begin{equation}\label{smooth-gradient-comparison}
\int_X|\nabla(u(t)-\varphi(T))|^2\omega(t)^n=-\int_X\big(u(t)-\varphi(T)\big)\Delta\big(u(t)-\varphi(T)\big)\omega(t)^n\leq Ce^{-t}.
\end{equation}

\textbf{Step 3: Letting $T\to\infty$.} Fix $t$ and let $T\to\infty$. By \eqref{limit-of-u} and \eqref{stc1}, $\varphi(T)$ converges pointwise on $X^o$ to $f^*\varphi_{can}$ and is uniformly bounded. Since $X\setminus X^o$ is null with respect to $\omega(t)^n$, we infer that
$$u(t)-\varphi(T)\longrightarrow u(t)-f^*\varphi_{can}
\qquad\text{in }L^2(X,\omega(t)^n). $$
By \eqref{smooth-gradient-comparison} and the lower semicontinuity of the Dirichlet integral with respect to $L^2$-convergence, the smooth function $u(t)-f^*\varphi_{can}$ on $X^o$ satisfies
$$
\int_{X^o}|\nabla(u(t)-f^*\varphi_{can})|^2\omega(t)^n\leq Ce^{-t}.
$$
Moreover, $\tr_{\omega(t)}\omega(T)\to\tr_{\omega(t)}f^*\omega_{can}$ pointwise on $X^o$. Applying Fatou's lemma to \eqref{smooth-trace-comparison} yields
$$
\int_{X^o}\big|\tr_{\omega(t)}f^*\omega_{can}-n\big|\omega(t)^n\leq Ce^{-t}.
$$
The required estimate \eqref{canonical-energy} follows.
\end{proof}

\subsection{Rate of convergence of scalar curvature}
We are ready to prove Theorem \ref{scalar-curvature-convergence} in the noncollapsing case.

{\bf Proof of Theorem \ref{scalar-curvature-convergence} for $\kappa=n$}: The lower bound of $R(t)+n$ and the integral estimate in \eqref{rate-of-R} are given by Proposition \ref{noncollapsing-integrals}. It remains to prove
\begin{equation}\label{upper-rate-of-R}
\sup_X\big(R(t)+n\big)\leq\frac{C}{1+t},\qquad\forall~t\geq0.
\end{equation}
Fix a positive multiple $k$ of $k_0$ and $y\in X_{can}$. Let $\rho=e^{-bk}$, $\psi$, $v$ and $F=R+n+|\nabla v|^2$ be as in Lemma \ref{local-canonical-potential} and Lemma \ref{local-scalar-lemma}. Denote by $C_0$ the constant in \eqref{local-canonical-gradient}. 

\textbf{Step 1: Integral estimate of $F$.} On $f^{-1}(B_{FS}(y,4\rho))\cap X^o$, we have
\begin{eqnarray*}
F&\leq&R+n+2|\nabla(u-f^*\varphi_{can})|^2+2|f^*\partial(\varphi_{can}+\psi)|^2_{\omega(t)}\\
&\leq&R+n+2|\nabla(u-f^*\varphi_{can})|^2+\frac{2C_0}{k}\big(tr_{\omega(t)}f^*\omega_{can}-n\big)+\frac{2nC_0}{k}.
\end{eqnarray*}
Hence,
$$\left(F-\frac{2nC_0}{k}\right)_+\leq|R+n|+2|\nabla(u-f^*\varphi_{can})|^2+2C_0\big|tr_{\omega(t)}f^*\omega_{can}-n\big|\quad\mbox{on}~f^{-1}(B_{FS}(y,4\rho))\cap X^o.$$
Since $X\setminus X^o$ has measure zero, \eqref{noncollapsing-energy} and \eqref{canonical-energy} give
\begin{equation}\label{integral-of-F-rate}
\frac{1}{vol_{\omega(t)}(X)}\int_{f^{-1}(B_{FS}(y,3\rho))}\Big(F(t)-\frac{2nC_0}{k}\Big)_+\omega(t)^n\leq Ce^{-t},\qquad\forall~t\geq0,
\end{equation}
where $C$ is independent of $k$ and $y$.

\textbf{Step 2: Upper bound of scalar curvature.} Applying \eqref{local-scalar-comparison} with $a=\frac{2nC_0}{k}$ and using \eqref{integral-of-F-rate}, we get for $0<h\leq1$ and $t\geq1$
$$\sup_{f^{-1}(B_{FS}(y,\rho))}\big(R(t)+n\big)\leq e^h\Big(\frac{2nC_0}{k}+C\cdot h^{-n}\cdot e^{-(t-h)}+C\cdot e^{C_1k}\cdot h\Big).$$
The constants are independent of $y\in X_{can}$. Taking the supremum over $y$, we get
\begin{equation}\label{quantitative-scalar-comparison}
\sup_X\big(R(t)+n\big)\leq C\Big(\frac{1}{k}+h^{-n}e^{-t}+e^{C_1k}h\Big),\qquad0<h\leq1,~t\geq1,
\end{equation}
for any positive multiple $k$ of $k_0$, where $C$ and $C_1$ are independent of $k,h,t$.

Let $t\geq8nC_1k_0$. Choose $h=e^{-\frac{t}{2n}}$ and let $k$ be the largest multiple of $k_0$ which is not greater than $\frac{t}{4nC_1}$, so that $\frac{t}{8nC_1}\leq k\leq\frac{t}{4nC_1}$. Then
$$\frac{1}{k}\leq\frac{8nC_1}{t},\qquad h^{-n}e^{-t}=e^{-\frac{t}{2}},\qquad e^{C_1k}h\leq e^{\frac{t}{4n}-\frac{t}{2n}}=e^{-\frac{t}{4n}}.$$
Inserting these into \eqref{quantitative-scalar-comparison} we get
$$\sup_X(R(t)+n)\leq\frac{C}{t}~~ \text{for} ~~t\geq8nC_1k_0.$$
Combining this estimate with \eqref{stc1} for $t\leq8nC_1k_0$ yields \eqref{upper-rate-of-R}. The proof of Theorem \ref{scalar-curvature-convergence} is completed. \qed




\begin{thebibliography}{99}
\bibitem{Bam18}
Richard H. Bamler. \emph{Convergence of Ricci flows with bounded scalar curvature.} Ann. of Math. (2) 188 (2018), no. 3, 753--831.

\bibitem{Bam20}
Richard H. Bamler. \emph{Entropy and heat kernel bounds on a Ricci flow background.} arXiv:2008.07093.

\bibitem{BamZh17}
Richard H. Bamler, Qi S. Zhang. \emph{Heat kernel and curvature bounds in Ricci flows with bounded scalar curvature.} Adv. Math. 319 (2017), 396--450.

\bibitem{Cao85}
Huai-Dong Cao. \emph{Deformation of K\"ahler metrics to K\"ahler-Einstein metrics on compact K\"ahler manifolds.} Invent. Math. 81 (1985), no. 2, 359--372.

\bibitem{CMZ23}
Pak-Yeung Chan, Zilu Ma, Yongjia Zhang. \emph{A local Sobolev inequality on Ricci flow and its applications.} J. Funct. Anal. 285 (2023), no. 5, Paper No. 109995, 36 pp.

\bibitem{CMZ22}
Pak-Yeung Chan, Zilu Ma, Yongjia Zhang. \emph{A uniform Sobolev inequality for ancient Ricci flows with bounded Nash entropy.} Int. Math. Res. Not. IMRN 2023, no. 13, 11127--11144.

\bibitem{CW12}
Xiuxiong Chen, Bing Wang. \emph{Space of Ricci flows I.} Comm. Pure Appl. Math. 65 (2012), no. 10, 1399--1457.

\bibitem{CW20}
Xiuxiong Chen, Bing Wang. \emph{Space of Ricci flows (II)-Part B: Weak compactness of the flows.} J. Differential Geom. 116 (2020), no. 1, 1--123.

\bibitem{CLee23}
Jianchun Chu, Man-Chun Lee. \emph{On the H\"older estimate of K\"ahler-Ricci flow.} Int. Math. Res. Not. IMRN 2023, no. 6, 4932--4951.

\bibitem{FoZh15}
Frederick Tsz-Ho Fong, Zhou Zhang. \emph{The collapsing rate of the K\"ahler-Ricci flow with regular infinite time singularity.} J. Reine Angew. Math. 703 (2015), 95--113.

\bibitem{FGS20}
Xin Fu, Bin Guo, Jian Song. \emph{Geometric estimates for complex Monge-Amp\`{e}re equations.} J. Reine Angew. Math. 765 (2020), 69--99.

\bibitem{GTZ13}
Mark Gross, Valentino Tosatti, Yuguang Zhang. \emph{Collapsing of abelian fibered Calabi-Yau manifolds.} Duke Math. J. 162 (2013), no. 3, 517--551.

\bibitem{GTZ16}
Mark Gross, Valentino Tosatti, Yuguang Zhang. \emph{Gromov-Hausdorff collapsing of Calabi-Yau manifolds.} Comm. Anal. Geom. 24 (2016), no. 1, 93--113.

\bibitem{GTZ20}
Mark Gross, Valentino Tosatti, Yuguang Zhang. \emph{Geometry of twisted K\"ahler-Einstein metrics and collapsing.} Comm. Math. Phys. 380 (2020), no. 3, 1401--1438.

\bibitem{GuPhSoSt24-1}
Bin Guo, Duong H. Phong, Jian Song, Jacob Sturm. \emph{Diameter estimates in K\"ahler geometry.} Comm. Pure Appl. Math. 77 (2024), no. 8, 3520--3556.

\bibitem{GS21}
Bin Guo, Jian Song. \emph{Positivity of Weil-Petersson currents on canonical models.} Pure Appl. Math. Q. 17 (2021), no. 3, 1045--1059.

\bibitem{HJST24}
Max Hallgren, Wangjian Jian, Jian Song, Gang Tian. \emph{Geometric regularity of blow-up limits of the K\"ahler-Ricci flow.} Geom. Funct. Anal. 34 (2024), no. 6, 1899--1972.

\bibitem{HLT24+}
Hans-Joachim Hein, Man-Chun Lee, Valentino Tosatti. \emph{Collapsing immortal K\"ahler-Ricci flows.} Forum Math. Pi 13 (2025), Paper No. e18, 98 pp.

\bibitem{HeTo15}
Hans-Joachim Hein, Valentino Tosatti. \emph{Remarks on the collapsing of torus fibered Calabi-Yau manifolds.} Bull. Lond. Math. Soc. 47 (2015), no. 6, 1021--1027.

\bibitem{Jian20}
Wangjian Jian. \emph{Convergence of scalar curvature of K\"ahler-Ricci flow on manifolds of positive Kodaira dimension.} Adv. Math. 371 (2020), 107253, 27 pp.

\bibitem{JS22}
Wangjian Jian, Jian Song. \emph{Diameter estimates for long-time solutions of the K\"ahler-Ricci flow.} Geom. Funct. Anal. 32 (2022), no. 6, 1335--1356.

\bibitem{JST23+}
Wangjian Jian, Jian Song, Gang Tian. \emph{Finite time singularities of the K\"ahler-Ricci flow.} arXiv:2310.07945.

\bibitem{LTZ26+}
Man-Chun Lee, Valentino Tosatti, Junsheng Zhang. \emph{Gromov-Hausdorff limits of immortal K\"ahler-Ricci flows.} arXiv:2602.19913v2.

\bibitem{Li07}
Junfang Li. \emph{Eigenvalues and energy functionals with monotonicity formulae under Ricci flow.} Math. Ann. 338 (2007), no. 4, 927--946.

\bibitem{LuWa20}
Yu Li, Bing Wang. \emph{Heat kernel on Ricci shrinkers.} Calc. Var. Partial Differential Equations 59 (2020), no. 6, Paper No. 194, 84 pp.

\bibitem{Ro81}
Oscar S. Rothaus. \emph{Logarithmic Sobolev inequalities and the spectrum of Schr\"odinger operators.} J. Functional Analysis 42 (1981), no. 1, 110--120.

\bibitem{SesT08}
Natasa Sesum, Gang Tian. \emph{Bounding scalar curvature and diameter along the K\"ahler Ricci flow (after Perelman).} J. Inst. Math. Jussieu 7 (2008), no. 3, 575--587.

\bibitem{So14}
Jian Song. \emph{Finite-time extinction of the K\"ahler-Ricci flow.} Math. Res. Lett. 21 (2014), no. 6, 1435--1449.

\bibitem{SSW13}
Jian Song, G\'abor Sz\'ekelyhidi, Ben Weinkove. \emph{The K\"ahler-Ricci flow on projective bundles.} Int. Math. Res. Not. IMRN 2013, no. 2, 243--257.

\bibitem{ST07}
Jian Song, Gang Tian. \emph{The K\"ahler-Ricci flow on surfaces of positive Kodaira dimension.} Invent. Math. 170 (2007), no. 3, 609--653.

\bibitem{ST12}
Jian Song, Gang Tian. \emph{Canonical measures and K\"ahler-Ricci flow.} J. Amer. Math. Soc. 25 (2012), no. 2, 303--353.

\bibitem{ST16}
Jian Song, Gang Tian. \emph{Bounding scalar curvature for global solutions of the K\"ahler-Ricci flow.} Amer. J. Math. 138 (2016), no. 3, 683--695.

\bibitem{ST17}
Jian Song, Gang Tian. \emph{The K\"ahler-Ricci flow through singularities.} Invent. Math. 207 (2017), no. 2, 519--595.

\bibitem{STZ19}
Jian Song, Gang Tian, Zhenlei Zhang. \emph{Collapsing behavior of Ricci-flat K\"ahler metrics and long time solutions of the K\"ahler-Ricci flow.} arXiv:1904.08345.

\bibitem{SW13-1}
Jian Song, Ben Weinkove. \emph{An introduction to the K\"ahler-Ricci flow.} An introduction to the K\"ahler-Ricci flow, 89--188, Lecture Notes in Math., 2086, Springer, Cham, 2013.

\bibitem{SW13}
Jian Song, Ben Weinkove. \emph{Contracting exceptional divisors by the K\"ahler-Ricci flow.} Duke Math. J. 162 (2013), no. 2, 367--415.

\bibitem{Sz25+}
G\'abor Sz\'ekelyhidi. \emph{Gromov-Hausdorff limits of collapsing Calabi-Yau fibrations.} arXiv:2505.14939.

\bibitem{Sz25}
G\'abor Sz\'ekelyhidi. \emph{Singular K\"ahler-Einstein metrics and RCD spaces.} Forum Math. Pi 13 (2025), Paper No. e24, 33 pp.

\bibitem{TiZh15}
Gang Tian, Qi S. Zhang. \emph{A compactness result for Fano manifolds and K\"ahler Ricci flows.} Math. Ann. 362 (2015), no. 3-4, 965--999.

\bibitem{TZhaZ16}
Gang Tian, Zhenlei Zhang. \emph{Convergence of K\"ahler-Ricci flow on lower-dimensional algebraic manifolds of general type.} Int. Math. Res. Not. IMRN 2016, no. 21, 6493--6511.

\bibitem{TZhZ16}
Gang Tian, Zhenlei Zhang. \emph{Regularity of K\"ahler-Ricci flows on Fano manifolds.} Acta Math. 216 (2016), no. 1, 127--176.

\bibitem{TZhZ21}
Gang Tian, Zhenlei Zhang. \emph{Relative volume comparison of Ricci flow.} Sci. China Math. 64 (2021), no. 9, 1937--1950.

\bibitem{TZh06}
Gang Tian, Zhou Zhang. \emph{On the K\"ahler-Ricci flow on projective manifolds of general type.} Chinese Ann. Math. Ser. B 27 (2006), no. 2, 179--192.

\bibitem{To18}
Valentino Tosatti. \emph{KAWA lecture notes on the K\"ahler-Ricci flow.} Ann. Fac. Sci. Toulouse Math. (6) 27 (2018), no. 2, 285--376.

\bibitem{TWY18}
Valentino Tosatti, Ben Weinkove, Xiaokui Yang. \emph{The K\"ahler-Ricci flow, Ricci-flat metrics and collapsing limits.} Amer. J. Math. 140 (2018), no. 3, 653--698.

\bibitem{TosZh15}
Valentino Tosatti, Yuguang Zhang. \emph{Infinite-time singularities of the K\"ahler-Ricci flow.} Geom. Topol. 19 (2015), no. 5, 2925--2948.

\bibitem{TosZh18}
Valentino Tosatti, Yuguang Zhang. \emph{Finite time collapsing of the K\"ahler-Ricci flow on threefolds.} Ann. Sc. Norm. Super. Pisa Cl. Sci. (5) 18 (2018), no. 1, 105--118.

\bibitem{Tsu88}
Hajime Tsuji. \emph{Existence and degeneration of K\"ahler-Einstein metrics on minimal algebraic varieties of general type.} Math. Ann. 281 (1988), no. 1, 123--133.

\bibitem{Wan18}
Bing Wang. \emph{The local entropy along Ricci flow Part A: the no-local-collapsing theorems.} Camb. J. Math. 6 (2018), no. 3, 267--346.

\bibitem{ZJS25+}
Junsheng Zhang. \emph{Weak transcendental base-point freeness and diameter lower bounds for the K\"ahler-Ricci flow.} arXiv:2511.14735.

\bibitem{ZhaZh23}
Lei Zhang, Zhenlei Zhang. \emph{On the finite time collapsing rate of K\"ahler-Ricci flow on projective bundles.} Proc. Amer. Math. Soc. 151 (2023), no. 12, 5385--5389.

\bibitem{ZhZh25-2}
Lei Zhang, Zhenlei Zhang. \emph{Complex Monge-Amp\`ere equation in Orlicz space and diameter bound.} arXiv:2601.09893.

\bibitem{ZhQ16}
Qi S. Zhang. \emph{Sobolev inequalities, heat kernels under Ricci flow, and the Poincar\'{e} conjecture.} CRC Press, Boca Raton, FL, 2011. x+422 pp. ISBN: 978-1-4398-3459-6.

\bibitem{ZhY19}
Yashan Zhang. \emph{Collapsing limits of the K\"ahler-Ricci flow and the continuity method.} Math. Ann. 374 (2019), no. 1-2, 331--360.

\bibitem{ZhZ09}
Zhou Zhang. \emph{Scalar curvature bound for K\"ahler-Ricci flows over minimal manifolds of general type.} Int. Math. Res. Not. IMRN 2009, no. 20, 3901--3912.

\end{thebibliography}
\end{document}